\documentclass[11pt]{amsart}

\usepackage{enumerate}
\usepackage{amsmath,amsthm,verbatim,amssymb,amsfonts,amscd, graphicx}
\usepackage{tikz-cd} 
\usepackage{graphics}
\usepackage{hyperref}
\usepackage{mathrsfs}
\usepackage{relsize}
\usepackage[sort cites=true, backend=biber]{biblatex}
\theoremstyle{plain}

\usepackage{biblatex}

\begin{document}

\title{Generating sequences and semigroups of valuations on 2-dimensional normal local rings}
\author{Arpan Dutta}
\def\NZQ{\mathbb}               
\def\NN{{\NZQ N}}
\def\QQ{{\NZQ Q}}
\def\ZZ{{\NZQ Z}}
\def\RR{{\NZQ R}}
\def\CC{{\NZQ C}}
\def\AA{{\NZQ A}}
\def\BB{{\NZQ B}}
\def\PP{{\NZQ P}}
\def\FF{{\NZQ F}}
\def\GG{{\NZQ G}}
\def\HH{{\NZQ H}}
\def\P{\mathcal P}

%
%
\let\union=\cup
\let\sect=\cap
\let\dirsum=\oplus
\let\tensor=\otimes
\let\iso=\cong
\let\Union=\bigcup
\let\Sect=\bigcap
\let\Dirsum=\bigoplus
\let\Tensor=\bigotimes

\newtheorem{Theorem}{Theorem}[section]
\newtheorem{Lemma}[Theorem]{Lemma}
\newtheorem{Corollary}[Theorem]{Corollary}
\newtheorem{Proposition}[Theorem]{Proposition}
\newtheorem{Remark}[Theorem]{Remark}
\newtheorem{Remarks}[Theorem]{Remarks}
\newtheorem{Example}[Theorem]{Example}
\newtheorem{Examples}[Theorem]{Examples}
\newtheorem{Definition}[Theorem]{Definition}
\newtheorem{Problem}[Theorem]{}
\newtheorem{Conjecture}[Theorem]{Conjecture}
\newtheorem{Question}[Theorem]{Question}
\let\epsilon\varepsilon
\let\phi=\varphi
\let\kappa=\varkappa

\def \a {\alpha}
\def \s {\sigma}
\def \d {\delta}
\def \g {\gamma}

%
%
\textwidth=15cm \textheight=22cm \topmargin=0.5cm
\oddsidemargin=0.5cm \evensidemargin=0.5cm \pagestyle{plain}

\address{Department of Mathematics,
	University of Missouri, Columbia, MO, 65211, USA}
\email{adk3c@mail.missouri.edu}

\begin{abstract}
	In this paper we develop a method for constructing generating sequences for valuations dominating the ring of a two dimensional quotient singularity. Suppose that $K$ is an algebraically closed field of characteristic zero,  $K[X,Y]$ is a polynomial ring over $K$
	and $\nu$ is a rational rank 1 valuation of the field $K(X,Y)$ which dominates $K[X,Y]_{(X,Y)}$. Given a  finite Abelian group $H$ acting diagonally on $K[X,Y]$, and a generating sequence of $\nu$ in $K[X,Y]$ whose members are eigenfunctions for the action of $H$,  we compute a generating sequence for the invariant ring $K[X,Y]^H$. We use this to compute the semigroup $S^{K[X,Y]^H}(\nu)$ of values of elements of $K[X,Y]^H$. We further determine when $S^{K[X,Y]}(\nu)$ is a finitely generated $S^{K[X,Y]^H}(\nu)$-module. 
\end{abstract}

\maketitle

\section*{Notations} Let $\NN$ denotes the natural numbers $\{ 0, 1, 2, \cdots \}$. We denote the positive integers by $\ZZ_{> 0}$ and the positive rational numbers by $\QQ_{> 0}$. If the greatest common divisor of two positive integers $a$ and $b$ is $d$, this is denoted by $(a,b) = d$. If $\{ \gamma_k \}_{k \geq 0}$ is a set of rational numbers, we define $G(\gamma_0 , \cdots , \gamma_n) = \sum_{k = 0}^{n} \gamma_k \ZZ$ and $G(\gamma_0 , \gamma_1, \cdots ) = \sum_{ k \geq 0} \gamma_k \ZZ$. Similarly we define $S(\gamma_0, \cdots , \gamma_n) = \sum_{k = 0}^{n} \gamma_k \NN$ and $S(\gamma_0, \gamma_1, \cdots ) = \sum_{ k \geq 0} \gamma_k \NN$. If a group $G$ is generated by $g_1, \cdots , g_n$, we denote this by $G = <g_1, \cdots , g_n>$.

\section*{Introduction}

Let $R$ be a local domain with maximal ideal $m_R$ and quotient field $L$, and $\nu$ be a valuation  of $K$ which dominates $R$.  Let $V_{\nu}$ be the valuation ring of $\nu$, with maximal ideal $m_{\nu}$ and $\Phi_{\nu}$ be the valuation group of $\nu$. The associated graded ring of $R$ along the valuation $\nu$, defined by Teissier in [\ref{T1}] and [\ref{T2}], is 
\begin{equation}\label{eq1}
{\rm gr}_{\nu}(R)=\bigoplus_{\gamma\in \Phi_{\nu}}\mathcal P_{\gamma}(R)/\mathcal P^+_{\gamma}(R)
\end{equation}
where
$$
\mathcal P_{\gamma}(R)=\{f\in R\mid \nu(f)\ge \gamma\}\mbox{ and }\mathcal P^+_{\gamma}(R)=\{f\in R\mid \nu(f)> \gamma\}.
$$ 
In general, ${\rm gr}_{\nu}(R)$  is not Noetherian. 
The  valuation semigroup  of $\nu$ on $R$ is 
\begin{equation}\label{eq2}
S^R(\nu)=\{\nu(f)\mid f\in R\setminus(0)\}.
\end{equation}
If $R/m_R=V_{\nu}/m_{\nu}$ then ${\rm gr}_{\nu}(R)$ is the group algebra of $S^R(\nu)$ over $R/m_R$, so that ${\rm gr}_{\nu}(R)$ is completely determined by $S^R(\nu)$.

A generating sequence of $\nu$ in $R$ is a set of elements of $R$ whose classes in ${\rm gr}_{\nu}(R)$ generate ${\rm gr}_{\nu}(R)$ as an $R/m_R$-algebra.  An important problem is to construct a generating sequence of $\nu$ in $R$  which gives explicit formulas for the value of an arbitrary element of $R$, and gives explicit computations of the algebra (\ref{eq1}) and the semigroup (\ref{eq2}).  For regular local rings $R$ of dimension 2, the construction of generating sequences is realized in a very satisfactory way by Spivakovsky [\ref{S}] (with the assumption that $R/m_R$ is algebraically closed) and by Cutkosky and Vinh [\ref{CV1}] for arbitrary regular local rings of dimension 2.  A consequence of this theory is a simple classification of the semigroups which occur as a valuation semigroup on a regular local ring of dimension 2. There has been some success in constructing generating sequences in Noetherian local rings of dimension $\ge 3$, for instance in [\ref{K}], [\ref{Mo}], [\ref{NS}] and [\ref{T2}], but the general situation is very complicated and is not well understood. 

Another direction is to construct generating sequences in normal 2 dimensional Noetherian local rings. This is also extremely difficult. In Section 9 of [\ref{CV1}], a generating sequence is constructed for a rational rank 1 non discrete valuation in the ring $R=k[u,v,w]/(uv-w^2)$, from which the semigroup is constructed. The example shows that the valuation semigroups of  valuations dominating a normal two dimensional Noetherian local ring  are much more complicated than those of valuations dominating a two dimensional regular local ring. 
In this paper, we develop the method of this example into a general theory. 

If $R$ is a 2 dimensional Noetherian local domain, and $\nu$ is a valuation of the quotient field $L$ of $R$ which dominates $R$, It follows from Abhyankar's inequality [\ref{Ab1}] that  the valuation group $\Phi_{\nu}$ of $\nu$ is a finitely generated group, except in the case when the rational rank of $\nu$ is 1 ($\Phi_{\nu}\otimes\QQ\cong \QQ$) and $\Phi_{\nu}$ is non discrete.  
As this is the essentially difficult case in dimension 2, we will restrict to such valuations. 

Let $K$ be an algebraically closed field of characteristic 0 and  $K[X,Y]$ be a polynomial ring in two variables, which has the maximal ideal $\mathfrak m=(x,y)$. Let $\alpha\in K$ be a primitive $m$-th root of unity and $\beta\in K$ be a primitive $n$-th root of unity.  Now the group ${\NZQ U}_m\times {\NZQ U}_n$ acts on $K[X,Y]$ by $K$-algebra isomorphisms, where
$$
(\alpha^i,\beta^j)X=\alpha^iX\mbox{ and }(\alpha^i,\beta^j)Y=\beta^jY.
$$
In Theorem  1.2, we give a classification of the  subgroups $H_{i,j,t,x}$ of ${\NZQ U}_m\times{\NZQ U}_n$.  Let
$$
A_{i,j,t,x}=K[X,Y]^{H_{i,j,t,x}}\mbox{ and }\mathfrak n=\mathfrak m\cap A_{i,j,t,x}.
$$
We say that $f\in K[X,Y]$ is an eigenfunction for the action of $H_{i,j,t,x}$ on $K[X,Y]$ if for all $g\in H_{i,j,t,x}$, $gf =\lambda_gf$ for some $\lambda_g\in K$.

Let $\nu$ be a rational rank 1 non discrete valuation dominating the local ring $K[X,Y]_{\mathfrak m}$.  Using the algorithm of \ref{S} or \ref{CV1}, we construct a generating sequence
\begin{equation}\label{eq3}
Q_0=X, Q_1=Y,Q_2,\ldots
\end{equation}
of $\nu$ in $K[X,Y]$. Let $\nu^*$ be the restriction of $\nu$ to the quotient field of $A_{i,j,t,x}$. In Theorem \ref{invariant subring theorem}, we construct a generating sequence of $\nu^*$ in $A_{i,j,t,x}$, when the  members of the generating sequence (\ref{eq3}) are eigenfunctions for the action of $H_{i,j,t,x}$ on $K[X,Y]$. We give an explicit construction of the valuation semigroups $S^{(A_{i,j,t,x})_{\mathfrak n}}(\nu)$ in Theorem \ref{invariant subring theorem}.

Suppose that a Noetherian local domain $B$ dominates a Noetherian local domain $A$. Let $L$ be the quotient field of $A$, $M$ be the quotient field of $B$ and suppose that $M$ is finite over $L$. 
Suppose that $\omega$ is a valuation of $L$ which dominates $A$ and $\omega^*$ is an extension of $\nu$ to $M$ which dominates $B$. We can ask if ${\rm gr}_{\omega^*}(B)$ is a finitely generated ${\rm gr}_{\omega}(A)$-module or if $S^{B}(\omega^*)$ is a finitely generated $S^{\omega^*}(B)$-module. In general, ${\rm gr}_{\omega^*}(B)$ is not a finitely generated ${\rm gr}_{\omega}(A)$-algebra, so is certainly not a finitely generated ${\rm gr}_{\omega}(A)$-module. However, is is shown in Theorem 1.5. [\ref{C2}] that if $A$ and $B$ are essentially of finite type over a field characteristic zero, then there exists a birational extension $A_1$ of $A$ and a birational extension $B_1$ of $B$ such that 
$\omega^*$ dominates $B_1$, $\omega$ dominates $A_1$, $B_1$ dominates $A_1$ and ${\rm gr}_{\omega^*}(B_1)$ is a finitely generated ${\rm gr}_{\omega}(A_1)$-module (so $S^{B_1}(\omega^*)$ is a finitely generated $S^{A_1}(\omega)$-module). 

The situation is much more subtle in positive characteristic and mixed characteristic. In Theorem 1 [\ref{C1}], it is shown that  If $A$ and $B$ are excellent of dimension two and $L\rightarrow M$ is separable, then there exist birational extension $A_1$ of $A$ and $B_1$ of $B$ such that $A_1$ and $B_1$ are regular, $B_1$ dominates $A_1$,  $\omega^*$ dominates $B_1$  and ${\rm gr}_{\omega^*}(B_1)$ is a finitely generated ${\rm gr}_{\omega}(A_1)$-algebra if and only if the valued field extension $L\rightarrow M$ is without defect. For a discussion of defect in a finite extension of valued fields, see [\ref{Ku1}].

In this paper, we completely answer the question of finite generation of $S^{[K[X,Y]_{\mathfrak m}}(\nu)$ as a  $S^{(A_{i,j,t,x})_{\mathfrak n}}(\nu)$-module  (and hence of ${\rm gr}_{\nu}(K[X,Y]_{\mathfrak m})$ as a  ${\rm gr}_{\nu}((A_{i,j,t,x})_{\mathfrak n})$-module) for valuations with a generating sequence of eigenfunctions.  We obtain the following results in Section 4.

\begin{Proposition}
	Let $R_{\mathfrak{m}} = K[X,Y]_{(X,Y)}$ and $H_{i,j,t,x}$ be a subgroup of $\mathbb{U}_m \times \mathbb{U}_n$. Let $\nu$ be a rational rank $1$ non discrete valuation $\nu$ dominating $R_{\mathfrak{m}}$ with a generating sequence (\ref{eq3}) of eigenfunctions for $H_{i,j,t,x}$. Then $S^{R_{\mathfrak{m}}} (\nu)$ is finitely generated over the subsemigroup $S^{(A_{i,j,t,x})_{\mathfrak{n}}} (\nu)$ if and only if $\exists N \in \mathbb{Z}_{> 0} $ such that $ Q_r \in A_{i,j,t,x} \, \forall \, r \geq N$. Further, if $Q_N \in A_{i,j,t,x}$, then $Q_M \in A_{i,j,t,x} \, \forall \, M \geq N \geq 1$. 
\end{Proposition}

\begin{Theorem} Let $R_{\mathfrak{m}} = K[X,Y]_{(X,Y)}$ and $H_{i,j,t,x}$ be a subgroup of $\mathbb{U}_m \times \mathbb{U}_n$.
	\begin{itemize}
		\item[1)] $\exists$ a rational rank $1$ non discrete valuation $\nu$ dominating $R_{\mathfrak{m}}$ with a generating sequence (\ref{eq3}) of eigenfunctions for $H_{i,j,t,x} \Longleftrightarrow (\frac{m}{i}, \frac{n}{j}) = t$. 
		\item[2)] If $(\frac{m}{i}, \frac{n}{j}) = t = 1$, then $S^{R_{\mathfrak{m}}} (\nu)$ is a finitely generated $S^{(A_{i,j,t,x})_{\mathfrak{n}}} (\nu)$-module for all rational rank $1$ non discrete valuations $\nu$ which dominate $R_{\mathfrak{m}}$ and have a generating sequence (\ref{eq3}) of eigenfunctions for $H_{i,j,t,x}$.
		\item[3)] If $(\frac{m}{i}, \frac{n}{j}) = t > 1$, then $S^{R_{\mathfrak{m}}} (\nu)$ is not a finitely generated $S^{(A_{i,j,t,x})_{\mathfrak{n}}} (\nu)$-module for all rational rank $1$ non discrete valuations $\nu$ which dominate $R_{\mathfrak{m}}$ and have a generating sequence (\ref{eq3}) of eigenfunctions for $H_{i,j,t,x}$. 
	\end{itemize}
\end{Theorem} 

In Section 5, we show that for the valuations we consider,  the restriction of $\nu$ to the quotient field of $A_{i,j,t,x}$ does not split in $K[X,Y]_{\mathfrak m}$. The failure of non splitting can be an obstruction to finite generation of $S^{\omega^*}(B)$ as an  $S^{\omega}(A)$-module (Theorem 5 [\ref{C1}]), but our result shows that it is not a sufficient condition.

\section{Subgroups of $U_m \times U_n$}Let $K$ be an algebraically closed field of characteristic zero. Let $\alpha$ be a primitive $m$-th root of unity, and $\beta$ be a primitive $n$-th root of unity, in $K$. We denote $\mathbb{U}_m = <\alpha>$, and $\mathbb{U}_n = <\beta>$, which are multiplicative cyclic groups of orders $m$ and $n$ respectively. 

\begin{Lemma}[Goursat]  Let $A$ and $B$ be two groups. There is a bijective correspondence between subgroups $G \leq A \times B$, and 5-tuples $\{ \overline{G_1}, G_1, \overline{G_2}, G_2, \theta\}$, where 
	\begin{equation*}
	G_1 \trianglelefteq \overline{G_1} \leq A \, , \,  G_2 \trianglelefteq \overline{G_2} \leq B \, , \,  \theta : \frac{\overline{G_1}}{G_1} \rightarrow \frac{\overline{G_2}}{G_2} \text{ is an isomorphism.}
	\end{equation*}
\end{Lemma}

\begin{proof}
	Let $\pi_1$ and $\pi_2$ denote the first and second projection maps respectively. Let $i_1 : A \rightarrow A \times B$ and $i_2 : B \rightarrow A \times B$ denote the inclusion maps. Given a subgroup $G$ of $A \times B$, we construct the elements of the 5-tuple as follows,
	\begin{align*}
	\overline{G_1} = \pi_1 (G), \, &  G_1 = i_1 ^{-1} (G) \\
	\overline{G_2} = \pi_2 (G) , \, & G_2 = i_2 ^{-1} (G) \\
	\theta : \frac{\overline{G_1}}{G_1} \rightarrow \frac{\overline{G_2}}{G_2} \text{ is defined by } &\theta (\overline{a}) = \overline{b}, \text{ if } (a,b) \in G.
	\end{align*} 
	
	By construction, $\overline{G_1} = \{ a \in A \, | \, \exists b \in B \text{ with } (a,b) \in G\}$ and $\textbf{}G_1 = \{ a \in A \, | \, (a,1) \in G\} $.
	Let $x \in G_1, a \in \overline{G_1}$. Then $(x,1) \in G$ and $(a,b) \in G$ for some $b \in B$ implies $(a,b)(x,1)(a,b)^{-1} \in G \Longrightarrow axa^{-1} \in G_1 \Longrightarrow G_1 \trianglelefteq \overline{G_1} .$ 
	Similarly, we have $G_2 \trianglelefteq \overline{G_2}$. 
	\newline Conversely suppose we are given a 5-tuple $\{ \overline{G_1}, G_1, \overline{G_2}, G_2, \theta\}$ satisfying the conditions of the theorem. Let $p : \overline{G_1} \times \overline{G_2} \rightarrow  \frac{\overline{G_1}}{G_1} \times \frac{\overline{G_2}}{G_2} $ be the natural surjection. Let $G_{\theta} < \frac{\overline{G_1}}{G_1} \times \frac{\overline{G_2}}{G_2}$ denote the graph of $\theta$. Then $G = p^{-1} (G_{\theta})$. 
	\par Now we show the bijectivity of the correspondence. First we establish injectivity. Suppose $G \neq H$ be two subgroups of $A \times B$, such that the corresponding 5-tuples are equal, if possible. Thus, $\{ \overline{G_1}, G_1, \overline{G_2}, G_2, \theta_G \} = \{ \overline{H_1}, H_1, \overline{H_2}, H_2, \theta_H \}$. Now $G \neq H \Longrightarrow \exists (a,b) \in G - H$, without loss of generality. But this contradicts $\theta_G = \theta_H$, since $\theta_G(\overline{a}) = \overline{b}$, but $\theta_H(\overline{a}) \neq \overline{b}$. So this correspondence is injective.
	\par Now we establish the surjectivity of the correspondence. Given a 5-tuple satisfying the conditions of the theorem, we construct a subgroup $G \leq A \times B$. Now, $G = p^{-1} (G_{\theta}) = \{ (g,h) \, | \, \overline{h} = \theta (\overline{g}), g \in \overline{G_1}, h \in \overline{G_2} \}$. $a \in \pi_1 (G) \Longrightarrow (a,b) \in G \text{ for some } b \in B \Longrightarrow a \in \overline{G_1}$. Conversely, $a \in \overline{G_1} \Longrightarrow \theta (\overline{a}) = \overline{b}$ for some $b \in \overline{G_2} \Longrightarrow (a,b) \in p^{-1} (G_{\theta}) = G \Longrightarrow a \in \pi_1 (G)$. Thus we have shown $\pi_1 (G) = \overline{G_1}$. Now, $a \in i_1 ^{-1} (G) \Longleftrightarrow (a,1) \in G = p^{-1} (G_{\theta}) \Longleftrightarrow p (a,1) = (\overline{a}, \overline{1}) \in G_{\theta} \Longleftrightarrow \theta (\overline{a}) = \overline{1} \Longleftrightarrow \overline{a} = \overline{1} \Longleftrightarrow a \in G_1$. Similarly we show, $\overline{G_2} = \pi_2 (G), G_2 = i_2 ^{-1} (G)$. 
	
\end{proof}

\begin{Theorem}\label{i,j,t,x theorem}
	Given positive integers $i,j,t,x$ satisfying the given conditions 
	\begin{equation*}
	i | m, \, j | n , \, t | \frac{m}{i}, \, t | \frac{n}{j} , \, (x,t) = 1, \, 1 \leq x \leq t  
	\end{equation*}
	let
	\begin{equation}\label{formula - i,j,t,x}
	H_{i,j,t,x} = \{ ( \alpha^{ai}, \beta^{bj} ) \, | \, b \equiv ax (\text{mod } t) \}.
	\end{equation}
	Then the $H_{i,j,t,x}$ are subgroups of $\mathbb{U}_m \times \mathbb{U}_n$. And given any subgroup $G$ of $\mathbb{U}_m \times \mathbb{U}_n$, there exist unique $i,j,t,x$ satisfying the above conditions such that $G = H_{i,j,t,x}$. 
\end{Theorem}

\begin{proof} We first show that the condition $b \equiv ax (\text{mod }t)$ is well defined under the given conditions on $i,j,t,x$. Suppose $(\alpha^{a_1 i}, \beta^{b_1 j}) = (\alpha^{a_2 i}, \beta^{b_2 j})$, that is, $a_1 i \equiv a_2 i (\text{mod }m)$,  and $b_1 j \equiv b_2 j (\text{mod }n)$. Then, $\frac{m}{i} \mid (a_1 - a_2)$ and $\frac{n}{j} \mid (b_1 - b_2)$. Thus, $t \mid (a_1 - a_2)$ and $t \mid (b_1 - b_2)$, hence $t \mid (b_1 - b_2) - (a_1 - a_2)x$. So, $[b_1 - a_1 x] \equiv [b_2 - a_2 x] (\text{mod }t)$. 
	\par{} We now show $H_{i,j,t,x}$ is a subgroup of $\mathbb{U}_m \times \mathbb{U}_n$. Taking $a = b = 0$, we have $(1,1) \in H_{i,j,t,x}$. Let  $( \alpha^{ai}, \beta^{bj} ), \, ( \alpha^{ci}, \beta^{dj} ) \in H_{i,j,t,x}$ be distinct elements. Then $b \equiv ax (\text{mod t})$, and $d \equiv cx (\text{mod t})$. Hence $(b - d) \equiv (a - c)x (\text{ mod t})$. So, $( \alpha^{(a-c)i}, \beta^{(b-d)j} ) = ( \alpha^{ai}, \beta^{bj} )( \alpha^{ci}, \beta^{dj} )^{-1} \in H_{i,j,t,x}$. Hence $H_{i,j,t,x}$ is a subgroup. 
	\par{} By Goursat's Lemma, the subgroups of $\mathbb{U}_m \times \mathbb{U}_n$ are in bijective correspondence with the 5-tuples $\{ \overline{G_1}, G_1, \overline{G_2}, G_2, \theta\}$, where $G_1 \trianglelefteq \overline{G_1} \leq \mathbb{U}_m \, , \,  G_2 \trianglelefteq \overline{G_2} \leq \mathbb{U}_n \, , \,  \theta : \frac{\overline{G_1}}{G_1} \simeq \frac{\overline{G_2}}{G_2}$. Now any subgroup of $\mathbb{U}_m = <\alpha>$ is of the form $H_i = <\alpha^i> = \mathbb{U}_{\frac{m}{i}}$, where $i|m$. Since $H_i$ is an abelian group, any subgroup is normal. Any subgroup of $H_i$ is of the form $H_{i t_i} = <\alpha^{i t_i}> = \mathbb{U}_{\frac{m}{i t_i}}$, where $t_i | \frac{m}{i}$. Similarly, any subgroup of $\mathbb{U}_n$ is of the form $H_j = <\beta^j> = \mathbb{U}_{\frac{n}{j}}$, where $j | n$. And any subgroup of $H_j$ is of the form $H_{j t_j} = <\beta^{j t_j}> = \mathbb{U}_{\frac{n}{j t_j}}$, where $t_j | \frac{n}{j}$. Now, $\frac{\mathbb{U}_{\frac{m}{i}}}{\mathbb{U}_{\frac{m}{i t_i}}} \simeq \mathbb{U}_{t_i}$ and $\frac{\mathbb{U}_{\frac{n}{j}}}{\mathbb{U}_{\frac{n}{j t_j}}} \simeq \mathbb{U}_{t_j}$. So, $\theta_{ij} : \frac{\mathbb{U}_{\frac{m}{i}}}{\mathbb{U}_{\frac{m}{i t_i}}} \simeq \frac{\mathbb{U}_{\frac{n}{j}}}{\mathbb{U}_{\frac{n}{j t_j}}} \Longleftrightarrow t_i = t_j$. Define $t = t_i = t_j$. Thus the subgroups of $\mathbb{U}_m \times \mathbb{U}_n$ are in bijective correspondence with the set of 5-tuples, 
	\begin{equation}\label{5 tuples properties}
	\begin{split}
	(<\alpha^{it}>, <\alpha^i>, <\beta^{jt}>, &<\beta^j>, \theta_{ij}) \\
	\text{where } i|m, \, j|n, \, t | \frac{m}{i}, \, t| \frac{n}{j} \text{ and } &\theta_{ij} : \frac{<\alpha^i>}{<\alpha^{it}>} \simeq \frac{<\beta^j>}{<\beta^{jt}>}. \\	\end{split}
	\end{equation}
	Any such isomorphism is given by $\theta_{ij} (\overline{\alpha^i}) = \overline{\beta^{xj}}$, where $(x,t) = 1, \, 1 \leq x \leq t$, and $\overline{v}$ denotes the residue of an element $v \in <\alpha^{i}>$ in $\frac{<\alpha^i>}{<\alpha^{it}>}$, or the residue of an element $v \in <\beta^{j}>$ in $\frac{<\beta^j>}{<\beta^{jt}>}$. 
	\newline If $G_{\theta_{ij}}$ denotes the graph of $\theta_{ij}$, then $G_{\theta_{ij}} = \{ (\overline{\alpha^{ri}}, \overline{\beta^{rxj}}) | \, r \in \mathbb{N}  \}$. Denoting the natural surjection $p : <\alpha^i> \times <\beta^j> \longrightarrow \frac{<\alpha^i>}{<\alpha^{it}>} \times \frac{<\beta^j>}{<\beta^{jt}>}$, we have
	\begin{align*}
	p^{-1} (G_{\theta_{ij}}) &= \{ (\alpha^{ai}, \beta^{bj}) \, | \, \alpha^{\overline{ai}} = \alpha^{\overline{ri}}, \beta^{\overline{bj}} = \beta^{\overline{rxj}}, \text{ for some } r \in \mathbb{N} \} \\
	&= \{ (\alpha^{ai}, \beta^{bj}) \, | \, \alpha^{(a-r)i} \in <\alpha^{it}>, \beta^{(b-rx)j} \in <\beta^{jt}> , \text{ for some } r \in \mathbb{N}\} \\
	&= \{ (\alpha^{ai}, \beta^{bj}) \, | \, a \equiv r (\text{mod t}), b \equiv rx (\text{mod t}), \text{ for some } r \in \mathbb{N} \}.
	\end{align*}
	We now show that, 
	\begin{equation}\label{a = r, b = rx}
	a \equiv r (\text{mod t}), b \equiv rx (\text{mod t}), \text{ for some } r \in \mathbb{N} \Longleftrightarrow b \equiv ax (\text{mod t}) .\\
	\end{equation} 
	If $a \equiv r (\text{mod t}), b \equiv rx (\text{mod t})$, then $a - r = td$ for some integer $d$. Then $b - ax = b - (td + r)x \equiv b - rx (\text{mod t}) \equiv 0 (\text{mod t}) \Longrightarrow b \equiv ax (\text{mod t})$. Conversely if $b \equiv ax (\text{mod t})$, and $a \equiv r (\text{mod t})$ for some $r$, then $ b \equiv rx (\text{mod t})$. Thus we have established (\ref{a = r, b = rx}). So, $p^{-1} (G_{\theta_{ij}}) = \{ (\alpha^{ai}, \beta^{bj}) \, | \, b \equiv ax (\text{mod t}) \}$. Thus we have that any subgroup of $\mathbb{U}_m \times \mathbb{U}_n$ is of the form 
	\begin{equation*}
	H_{i,j,t,x} = \{ (\alpha^{ai}, \beta^{bj}) \, | \, b \equiv ax (\text{mod t}) \, ; \, i | m, \, j | n , \, t | \frac{m}{i}, \, t | \frac{n}{j} , \, (x,t) = 1, \, 1 \leq x \leq t \}. 
	\end{equation*}
	\par{} We now establish uniqueness. Let $(i_1, j_1, t_1, x_1)$ and $(i_2, j_2, t_2, x_2)$ be two distinct quadruples satisfying the conditions of the theorem, such that $H_{i_1, j_1, t_1, x_1} = H_{i_2, j_2, t_2, x_2}$. 
	From (\ref{5 tuples properties}), we observe $H_{i_1, j_1, t_1, x_1} = H_{i_2, j_2, t_2, x_2} $ implies
	\begin{align*}
	(<\alpha^{i_1 t_1}>, <\alpha^{i_1}>, <\beta^{j_1 t_1}>, <\beta^{j_1}>, \theta_{i_1 j_1}^{(1)}) &= (<\alpha^{i_2 t_2}>, <\alpha^{i_2}>, <\beta^{j_2 t_2}>, <\beta^{j_2}>, \theta_{i_2 j_2}^{(2)}) .\\ 
	\end{align*}
	Now, $<\alpha^{i_1}> = <\alpha^{i_2}> \Longrightarrow  |<\alpha^{i_1}>| = |<\alpha^{i_2}>| \Longrightarrow  m/{i_1} = m/ {i_2} \Longrightarrow  i_1 = i_2 = i.$ And, $<\alpha^{i t_1}> = <\alpha^{i t_2}> = m/ {i t_1} = m/{i t_2} = t_1 = t_2 = t$.	Similarly $j = j_1 = j_2$. Now, $\theta_{ij}^{(1)} = \theta_{ij}^{(2)} \Longrightarrow \, \theta_{ij}^{(1)} (\overline{\alpha^i}) = \theta_{ij}^{(2)} (\overline{\alpha^i}) \Longrightarrow \, \overline{\beta^{x_1 j}} = \overline{\beta^{x_2 j}} \text{ in } \frac{<\beta^j>}{<\beta^{tj}>} $. Thus, $t \mid |x_1 - x_2|$. Since $0 < x_1, x_2 \leq t$, we have $| x_1 - x_2 | = 0$, i.e. $x_1 = x_2$. Let $x = x_1 = x_2$. Then $(i,j,t,x) = (i_1, j_1, t_1, x_1) = (i_2, j_2, t_2, x_2)$ is unique. 	
	
\end{proof}

\begin{Proposition}\label{order of group}
	Let $i,j,t,x$ be positive integers satisfying the conditions of Theorem \ref{i,j,t,x theorem} such that $(\frac{m}{i}, \frac{n}{j}) = t$. Write $\frac{m}{i} = Mt$ and $\frac{n}{j} = Nt$ where $M,N \in \mathbb{Z}_{> 0}$ and $(M,N) = 1$. Then $|H_{i,j,t,x}| = MNt$.
\end{Proposition}

\begin{proof}
	Recall, $H_{i,j,t,x} = \{ (\alpha^{ai}, \beta^{bj}) \mid b \equiv ax (\text{mod }t)  \}$. We observe, as elements of $H_{i,j,t,x}$, $(\alpha^{a_1 i}, \beta^{b_1 j}) = (\alpha^{a_2 i}, \beta^{b_2 j})$ if and only if $a_1 \equiv a_2 (\text{mod }Mt)$ and $b_1 \equiv b_2 (\text{mod }Nt)$. Thus every element of $H_{i,j,t,x}$ has an unique representation,
	\begin{align} 
	H_{i,j,t,x} = \{ (\alpha^{ai}, \beta^{bj}) \mid b \equiv ax (\text{mod }t), \, 0 \leq a < Mt, \, 0 \leq b < Nt  \}.
	\end{align}
	Hence there is a bijective correspondence,
	\begin{align*}
	H_{i,j,t,x} &\longleftrightarrow \{ (a,b) \mid b \equiv ax (\text{mod }t), \, 0 \leq a < Mt, \, 0 \leq b < Nt, \, a,b \in \ZZ \} \\
	&\longleftrightarrow \{ (a, ax + \lambda t) \mid 0 \leq a < Mt, \, 0 \leq ax+\lambda t < Nt, \, a, \lambda \in \ZZ \} \\
	&\longleftrightarrow \{ (a, \lambda ) \mid  0 \leq a < Mt, \, 0 \leq \lambda + \frac{ax}{t} < N, \, a, \lambda \in \ZZ \}.
	\end{align*}
	Hence there are $Mt$ possible choices for $a$. And for each choice of $a$, there are $N$ possible choices for $\lambda$. Thus $|H_{i,j,t,x}| = MNt$.
\end{proof}

\section{Generating Sequences}\label{GS}

\par{} In this section we establish notation which will be used throughout the paper. Let $R = K[X,Y]$ be a polynomial ring in two variables over an algebraically closed field $K$ of characteristic zero. Let $\mathfrak{m} = (X,Y)$ be the maximal ideal of $R$. Then $\mathbb{U}_m \times \mathbb{U}_n$ acts on $R$ by $K$-algebra isomorphisms satisfying 
\begin{equation}\label{K-algebra action}
(\alpha^x, \beta^y) \cdot ( X^r Y^s) = \alpha^{rx} \beta^{sy} X^r Y^s.
\end{equation}
Thus, $R^{H_{i,j,t,x}} = \{ \sum_{r,s} c_{r,s} X^r Y^s \in R \, | \, \alpha^{rai} \beta^{sbj} = 1 \, \forall \, r,s, \, \forall \, b \equiv ax (\text{mod t}) \}$. 
\newline $f \in R$ is defined to be an eigenfunction of $H_{i,j,t,x}$ if $(\alpha^{ai}, \beta^{bj}) \cdot f = \lambda_{ab} f$ for some $\lambda_{ab} \in K$, for all  $(\alpha^{ai}, \beta^{bj}) \in H_{i,j,t,x}$. The eigenfunctions of $H_{i,j,t,x}$ are of the form  $f = \mathlarger{\sum_{r,s}}c_{r,s} X^r Y^s \in R \text{ such that } \alpha^{rai} \beta^{sbj} \text{ is a common constant } \forall \, r,s \text{ such that } c_{r,s} \neq 0,  \forall \, b \equiv ax (\text{mod t})$.  

Let $\nu$ be a rational rank 1 non discrete valuation of $K(X,Y)$ which dominates $R_{\mathfrak{m}}$. The algorithm of Theorem $4.2$ of [\ref{CV1}] (as refined in Section $(8)$ of [\ref{CV1}]) produces a generating sequence 
\begin{equation}\label{Gen Seq}
Q_0 = X, Q_1 = Y, Q_2, \cdots 	
\end{equation}
of elements in $R$ which have the following properties. 
\begin{enumerate}
	\item[1)] Let $\gamma_l = \nu (Q_l) \, \forall \, l \geq 0$ and $\overline{ m_{l}} = [G(\gamma_0 , \cdots , \gamma_l) : G(\gamma_0 , \cdots , \gamma_{l-1})]$ = min $\{ q \in \ZZ_{> 0 } \mid  q \gamma_l \in  G(\gamma_0 , \cdots , \gamma_{l-1}) \} \, \forall \, l \geq 1$. Then $\gamma_{l+1} > \overline{ m_{l}} \gamma_l \, \forall \, l \geq 1$. 
	\item[2)] Set $d(l) $ = deg$_Y (Q_l) \, \forall \, l \in \ZZ_{> 0 }$. Then, $Q_l = Y^{d(l)} + Q_l^{*} (X,Y)$, where deg$_Y (Q_l^{*} (X,Y)) < d(l)$. We have that, $d(1) = 1$, $d(l) = \prod_{k=1}^{l-1}\overline{ m_k}  \, \forall \, l \geq 2$. In particular, $1 \leq l_1 \leq l_2 \Longrightarrow d(l_1) \mid d(l_2)$. 
	\item[3)] Every $f \in R$ with deg$_Y (f) = d$ has a unique expression 
	\begin{equation*}\label{expression of f}
	f = \mathlarger{\sum_{m = 0}^d}\mathlarger{[(\sum_l b_{l,m}X^l) Q_1 ^{j_1 (m)} \cdots Q_r ^{j_r (m)}]}
	\end{equation*}
	where $b_{l,m} \in K$, $0 \leq j_l (m) < \overline{m_l} \, \forall \, l \geq 1$, and $\text{deg}_Y [Q_1 ^{j_1 (m)} \cdots Q_r ^{j_r (m)}] = m \, \forall \, m$. Writing $f_m = (\sum_l b_{l,m}X^l) Q_1 ^{j_1 (m)} \cdots Q_r ^{j_r (m)}$, we have that $\nu (f_m) = \nu(f_n) \Longleftrightarrow m = n$. So, $\nu (f) = \text{min}_m \{\nu(f_m) \}$.
	\item[4)] From $3)$ we have that the semigroup $S^{R_{\mathfrak{m}}} (\nu) = \{\nu(f) \mid 0 \neq f \in R \} = S(\gamma_l \mid l \geq 0) $.
	
\end{enumerate}

Suppose that $\nu$ is a rational rank $1$ non discrete valuation dominating $R_{\mathfrak{m}}$. We will say that $\nu$ has a generating sequence of eigenfunctions for $H_{i,j,t,x}$ if all $Q_l$ in the generating sequence (\ref{Gen Seq}) of Section \ref{GS} are eigenfunctions for $H_{i,j,t,x}$.

\section{Valuation Semigroups of Invariant Subrings}

\begin{Theorem}\label{invariant subring theorem}
	Let $i,j,t,x$ be positive integers satisfying the conditions of Theorem \ref{i,j,t,x theorem}. Suppose that $\nu$ is a rational rank $1$ non discrete valuation dominating $R_{\mathfrak{m}}$, where $R = K[X,Y]$, and $\mathfrak{m} = (X,Y)$. Suppose that $\nu$ has a generating sequence (\ref{Gen Seq})
	\begin{equation*}
	Q_0 = X, Q_1 = Y, Q_2, \cdots
	\end{equation*}
	such that each $Q_l \in R$ is an eigenfunction for $H_{i,j,t,x}$. Let notation be as in Section \ref{GS}. Then denoting $A_{i,j,t,x} = R^{H_{i,j,t,x}}$, and defining $ \mathfrak{n} = \mathfrak{m} \cap A_{i,j,t,x}$ we have  
	\begin{equation}\label{invariant subsemigroup formula}
	S^{(A_{i,j,t,x})_{\mathfrak{n}}} (\nu)= \left\{ l \gamma_0 + j_1 \gamma_1 + \cdots + j_r \gamma_r  \, \left | \, \begin{array}{@{}l@{\thinspace}l} 
	l \in \mathbb{N}, \,  r \in \NN, \, 0 \leq j_k < \overline{m_k} \, \forall \, k = 1, \cdots , r \\
	\alpha^{lai}\beta^{bj\mathlarger{\sum_{k=1}^r [j_{k} d(k)] }} = 1\\
	\forall \, b \equiv ax (\text{mod t})\\
	\end{array}
	\right.
	\right\}.
	\end{equation}
\end{Theorem}

\begin{proof}
	Let $0 \neq f(X,Y) \in R$, with $\text{deg}_Y (f) = d$. By (\ref{K-algebra action}), $(\alpha^{ai}, \beta^{bj})\cdot Y^{d(m)} = \beta^{d(m)bj} Y^{d(m)}$. Since $Q_m$ is an eigenfunction of $H_{i,j,t,x}$, we conclude that for $m > 0$,
	\begin{equation}\label{Q_m as eigenfunction}
	(\alpha^{ai}, \beta^{bj})\cdot Q_m = \beta^{d(m)bj} Q_m = \beta^{\text{deg}_Y (Q_m) bj} Q_m\,, \, \forall \, (\alpha^{ai}, \beta^{bj}) \in H_{i,j,t,x}.
	\end{equation}
	We also have, $(\alpha^{ai}, \beta^{bj})\cdot Q_0 = (\alpha^{ai}, \beta^{bj})\cdot X = \alpha^{ai}X \,, \, \forall \, (\alpha^{ai}, \beta^{bj}) \in H_{i,j,t,x}$. Now $f$ has an expansion of the form $3)$ of Section \ref{GS}. So,
	\begin{align*}
	(\alpha^{ai}, \beta^{bj})\cdot f &= (\alpha^{ai}, \beta^{bj}) \cdot \mathlarger{\sum_{m = 0}^d}\mathlarger{[(\sum_l b_{l,m}X^l) Q_1 ^{j_1 (m)} \cdots Q_r ^{j_r (m)}]}\\
	&= \mathlarger{\sum_{m = 0}^d}\mathlarger{[(\sum_l \alpha^{lai} b_{l,m}X^l) \beta^{bj\mathlarger{\sum_{k=1}^r [j_{k} (m) d(k)] }} Q_1 ^{j_1 (m)} \cdots Q_r ^{j_r (m)}]}.\\
	\end{align*}	
	Now, $f \in A_{i,j,t,x} \Longleftrightarrow \alpha^{lai}\beta^{bj\mathlarger{\sum_{k=1}^r [j_{k} (m) d(k)]}} =1, \,  \forall \, b \equiv ax \text{( mod t)} , \, \forall \,l,$ such that $b_{l,m} \neq 0$. 
	\newline So,
	\begin{align*}
	\{ \nu(f) \, | \, 0 \neq f \in (A_{i,j,t,x})_{\mathfrak{n}} \} &= \{ \nu(f) \, | \, 0 \neq f \in A_{i,j,t,x} \}\\ 
	&\subset \left\{ l \gamma_0 + j_1 \gamma_1 + \cdots + j_r \gamma_r  \, \left | \, \begin{array}{@{}l@{\thinspace}l} 
	l \in \mathbb{N}, \, r \in \NN,  \, 0 \leq j_k < \overline{m_k} \, \forall \, k = 1, \cdots , r\\
	\alpha^{lai}\beta^{bj\mathlarger{\sum_{k=1}^r [j_{k} d(k)] }} = 1\\
	\forall \, b \equiv ax \text{( mod t )}\\
	\end{array}
	\right.
	\right\}.
	\end{align*}
	Conversely, suppose we have $l \in \mathbb{N}, \, r \in \NN, \,  0 \leq j_k < \overline{m_k} \, \forall \, k = 1, \cdots , r$ such that $ \forall \, b \equiv ax \text{(mod t)}$ we have $\alpha^{lai}\beta^{bj\mathlarger{\sum_{k=1}^r [j_{k} d(k)] }} = 1$. Define $f(X,Y) = X^l Q_{1} ^{j_1} \cdots Q_{r} ^{j_r} \in R $. For any $(\alpha^{ai}, \beta^{bj}) \in H_{i,j,t,x}$ we have, $(\alpha^{ai}, \beta^{bj}) \cdot f = (\alpha^{ai}, \beta^{bj}) \cdot (X^l Q_{1} ^{j_1} \cdots Q_{r} ^{j_r}) = \alpha^{lai}\beta^{bj\mathlarger{\sum_{k=1}^r [j_{k} d(k)] }} X^l Q_{1} ^{j_1} \cdots Q_{r} ^{j_r} = f$, that is, $f \in A_{i,j,t,x}$. So $\nu(f) = l \gamma_0 + j_1 \gamma_1 + \cdots + j_r \gamma_r \in S^{(A_{i,j,t,x})_{\mathfrak{n}}} (\nu)$. Hence we conclude,
	\begin{equation*}
	S^{(A_{i,j,t,x})_{\mathfrak{n}}} (\nu) = \left\{ l \gamma_0 + j_1 \gamma_1 + \cdots + j_r \gamma_r  \, \left | \, \begin{array}{@{}l@{\thinspace}l} 
	l \in \mathbb{N}, \, r \in \NN, \, 0 \leq j_k < \overline{m_k} \, \forall \, k = 1, \cdots , r\\
	\alpha^{lai}\beta^{bj\mathlarger{\sum_{k=1}^r [j_{k} d(k)] }} = 1\\
	\forall \, b \equiv ax \text{( mod t )}\\
	\end{array}
	\right.
	\right\}.
	\end{equation*}
\end{proof}

\section{Finite and Non-Finite Generation}	
In this section we study the finite and non-finite generation of the valuation semigroup $S^{R_{\mathfrak{m}}} (\nu)$ over the subsemigroup $S^{(A_{i,j,t,x})_{\mathfrak{n}}} (\nu)$. A semigroup $S$ is said to be finitely generated over a subsemigroup $T$ if there are finitely many elements $ s_1, \cdots , s_n $ in $S$ such that $S = \{ s_1, \cdots , s_n  \} + T$. 
\par At the end of this section we will prove the following theorem. 
\begin{Theorem}\label{Main Theorem} Let $R_{\mathfrak{m}} = K[X,Y]_{(X,Y)}$ and $H_{i,j,t,x}$ be a subgroup of $\mathbb{U}_m \times \mathbb{U}_n$.
	\begin{itemize}
		\item[1)] $\exists$ a rational rank $1$ non discrete valuation $\nu$ dominating $R_{\mathfrak{m}}$ with a generating sequence (\ref{Gen Seq}) of eigenfunctions for $H_{i,j,t,x} \Longleftrightarrow (\frac{m}{i}, \frac{n}{j}) = t$. 
		\item[2)] If $(\frac{m}{i}, \frac{n}{j}) = t = 1$, then $S^{R_{\mathfrak{m}}} (\nu)$ is a finitely generated $S^{(A_{i,j,t,x})_{\mathfrak{n}}} (\nu)$-module for all rational rank $1$ non discrete valuations $\nu$ which dominate $R_{\mathfrak{m}}$ and have a generating sequence (\ref{Gen Seq}) of eigenfunctions for $H_{i,j,t,x}$.
		\item[3)] If $(\frac{m}{i}, \frac{n}{j}) = t > 1$, then $S^{R_{\mathfrak{m}}} (\nu)$ is not a finitely generated $S^{(A_{i,j,t,x})_{\mathfrak{n}}} (\nu)$-module for all rational rank $1$ non discrete valuations $\nu$ which dominate $R_{\mathfrak{m}}$ and have a generating sequence (\ref{Gen Seq}) of eigenfunctions for $H_{i,j,t,x}$. 
	\end{itemize}
\end{Theorem}

\par  We introduce some notation. Let $\sigma(0) = 0, \, \sigma(l) = \text{min } \{ j \, | \, j > \sigma(l-1) \text{ and } \overline{ m_j} > 1  \}$. Let $P_l = Q_{\sigma(l)}$ and $\beta_l = \nu(P_l) = \gamma_{\sigma(l)} \, \forall \, l \geq 0$. Let $\overline{ n_l}  = [G(\beta_0, \cdots , \beta_l) : G(\beta_0, \cdots , \beta_{l-1} )] = \text{min} \{ q \in \mathbb{Z}_{> 0} \, | \, q \beta_l \in G(\beta_0, \cdots , \beta_{l-1} ) \} \, \forall \, l \geq 1$. Then $\overline{ n_l} = \overline{ m_{\sigma(l)}}$. $S^{R_{\mathfrak{m}}} (\nu) = S(\gamma_0 , \gamma_1 , \cdots ) = S(\beta_0 , \beta_1 , \cdots )$ and $\{ \beta_l \}_{l \geq 0}$ form a minimal generating set of $S^{R_{\mathfrak{m}}} (\nu)$, that is, $\overline{ n_l} > 1 \, \forall \, l \geq 1$. 
\par 
\vspace{0.2 mm} We first make a general observation. Suppose for some $d \geq  1$, $j_r \neq 0$ and $l , j_1 , \cdots , j_r \in \NN$, we have an expression of the form, $\beta_d = l \beta_0 + j_1 \beta_1 + \cdots + j_r \beta_r$. If $r > d$ then $j_r \beta_r \geq \beta_r > \beta_d$ which is a contradiction. If $r < d$ then $\beta_d \in G(\beta_0, \cdots , \beta_{d-1}) \Longrightarrow \overline{ n_d} = 1$. This is a contradiction as $\overline{ n_l} > 1 \, \forall \, l \geq 1$. Thus, $\beta_r = l \beta_0 + j_1 \beta_1 + \cdots + j_r \beta_r$. If $j_r > 1$, then $j_r \beta_r > \beta_r$. If $j_r = 0$, then $\beta_r \in G(\beta_0 , \cdots , \beta_{r-1}) \Longrightarrow \overline{n_r} = 1$. So, $j_r = 1$. Since $\beta_i > 0 \, \forall \, i$, we then have $l = 0, j_i = 0 \, \forall \, i \neq r$. Thus, for $l, j_1, \cdots , j_r \in \NN$ and $d \geq 1$, 
\begin{equation}\label{min gen relation}
\beta_d = l \beta_0 + j_1 \beta_1 + \cdots + j_r \beta_r \Longrightarrow j_d = 1, l = 0, j_i = 0 \, \forall \, i \neq d. 
\end{equation}

\begin{Proposition}\label{f.g iff all q_i after some i}
	Let $R_{\mathfrak{m}} = K[X,Y]_{(X,Y)}$ and $H_{i,j,t,x}$ be a subgroup of $\mathbb{U}_m \times \mathbb{U}_n$. Let assumptions be as in Theorem \ref{invariant subring theorem}. Then $S^{R_{\mathfrak{m}}} (\nu)$ is finitely generated over the subsemigroup $S^{(A_{i,j,t,x})_{\mathfrak{n}}} (\nu)$ if and only if $\exists N \in \mathbb{Z}_{> 0} $ such that $ Q_r \in A_{i,j,t,x} \, \forall \, r \geq N$. Further, if $Q_N \in A_{i,j,t,x}$, then $Q_M \in A_{i,j,t,x} \, \forall \, M \geq N \geq 1$. 
\end{Proposition}	

\begin{proof}
	We first show that, for any $r \geq 1$, $\gamma_r \in S^{(A_{i,j,t,x})_{\mathfrak{n}}} (\nu) \Longleftrightarrow Q_r \in A_{i,j,t,x}$. It is enough to show the implication $\gamma_r \in S^{(A_{i,j,t,x})_{\mathfrak{n}}} (\nu) \Longrightarrow Q_r \in A_{i,j,t,x}$. From (\ref{invariant subsemigroup formula}) we have, $\gamma_r \in S^{(A_{i,j,t,x})_{\mathfrak{n}}} (\nu) \Longrightarrow \gamma_r = l \gamma_0 + j_1 \gamma_1 + \cdots + j_s \gamma_s$, where $l \in \NN$, $s \in \NN$,  $0 \leq j_k < \overline{ m_k}$ and $\alpha^{lai} \beta^{bj \sum_{ k = 1}^{s} j_k d(k)} = 1 \, \forall \, b \equiv  ax (\text{mod }t) $.
	\newline Since $l, j_1, \cdots , j_s \in \NN$, $\gamma_i < \gamma_{i+1} \, \forall \, i \geq 1$ and $\gamma_i > 0 \, \forall \, i$, we have $r \geq s$. If $r = s$, then $\gamma_r = l \gamma_0 + \sum_{k = 1}^{r} j_k \gamma_k \geq j_r \gamma_r \geq \gamma_r$. Since $j_r \neq 0$ and $j_r \in \NN$ we have $j_r = 1$. And $\gamma_i > 0 \, \forall \, i$ implies $l = j_1 = \cdots = j_{r-1} = 0$. Then $\beta^{bj d(r)} = 1 \, \forall \, b \equiv  ax (\text{mod }t) $. So from (\ref{Q_m as eigenfunction}), $(\alpha^{ai}, \beta^{bj}) \cdot Q_r = Q_r \, \forall \, b \equiv  ax (\text{mod }t)$, that is, $Q_r \in A_{i,j,t,x}$.
	\par  If $r > s$, then $\gamma_r = l \gamma_0 + \sum_{k = 1}^{s} j_k \gamma_k \Longrightarrow \overline{ m_r} = 1$. Since $0 \leq j_k < \overline{ m_k}$, by Equation $(8)$ in [\ref{CV1}] we have $Q_{r+1} = Q_r - \lambda X^{l} Y^{j_1} Q_2^{j_2} \cdots Q_{s}^{j_{s}}$ where $\lambda \in K \setminus \{ 0 \}$. Since each $Q_m$ is an eigenfunction for $H_{i,j,t,x}$, from (\ref{Q_m as eigenfunction}) we have, $\forall \, b \equiv  ax (\text{mod }t)$, 
	\begin{equation*}
	\beta^{bj d(r+1)} Q_{r+1} = \beta^{bj d(r)} Q_r - \lambda \alpha^{lai} \beta^{bj \sum_{ k = 1}^{s} j_k d(k)} X^{l} Y^{j_1} Q_2^{j_2} \cdots Q_{s}^{j_{s}} .
	\end{equation*}
	Again by $2)$ in Section \ref{GS} we have $d(r+1) = \overline{ m_1} \cdots  \overline{ m_r} = \overline{ m_1} \cdots \overline{ m_{r-1}} = d(r)$, as $\overline{ m_r} = 1$. So the above expression yields $\beta^{bj d(r)} Q_{r+1} = \beta^{bj d(r)} Q_r - \lambda \alpha^{lai} \beta^{bj \sum_{ k = 1}^{s} j_k d(k)} X^{l} Y^{j_1} Q_2^{j_2} \cdots Q_{s}^{j_{s}} \, \forall \, b \equiv  ax (\text{mod }t)$. Since $Q_{r+1}$ is an eigenfunction, this implies $\beta^{bj d(r)} = \alpha^{lai} \beta^{bj \sum_{ k = 1}^{s} j_k d(k)} = 1 \, \forall \, b \equiv  ax (\text{mod }t) $. From (\ref{Q_m as eigenfunction}), we then have $Q_r \in A_{i,j,t,x}$.
	\par 
	\vspace{2 mm}
	To prove the proposition, we now show $S^{R_{\mathfrak{m}}} (\nu) $ is finitely generated over the subsemigroup $S^{(A_{i,j,t,x})_{\mathfrak{n}}} (\nu)$ if and only if $\exists N \in \ZZ_{> 0}$ such that $\forall \, r \geq N, \, \gamma_r \in S^{(A_{i,j,t,x})_{\mathfrak{n}}} (\nu)$. 
	\newline Suppose $S^{R_{\mathfrak{m}}} (\nu)$ is finitely generated over $S^{(A_{i,j,t,x})_{\mathfrak{n}}} (\nu)$. So, $\exists \, x_0, \cdots , x_l \in S^{R_{\mathfrak{m}}} (\nu)$ such that $S^{R_{\mathfrak{m}}} (\nu) = \{ x_0, \cdots , x_l \} + S^{(A_{i,j,t,x})_{\mathfrak{n}}} (\nu)$. Let $L \in \NN$ be the least natural number such that $ S^{R_{\mathfrak{m}}} (\nu) = S(\beta_0, \cdots , \beta_L) + S^{(A_{i,j,t,x})_{\mathfrak{n}}} (\nu)$, where $\beta_i = \gamma_{\sigma(i)} \, \forall \, i \geq 0$. Suppose, if possible, $\exists \, r > \sigma(L) \geq 0$ such that $\gamma_r \notin S^{(A_{i,j,t,x})_{\mathfrak{n}}} (\nu)$. Choose $M$ such that $\sigma(M) \leq r < \sigma(M+1)$. Then $\sigma(L) < \sigma(M)$, that is $L < M$. So $\beta_L < \beta_M \leq \gamma_r < \beta_{M+1}$. Now $\beta_M$ has an expression $\beta_M = \sum_{i = 0}^L a_i \beta_i + y$ where $y \in S^{(A_{i,j,t,x})_{\mathfrak{n}}} (\nu), \, a_i \in \mathbb{N}$. From (\ref{invariant subsemigroup formula}) we have $\beta_M = \sum_{i = 0}^L a_i \beta_i + (l \gamma_0 + j_1 \gamma_1 + \cdots + j_s \gamma_s)$, where $0 \leq j_k < \overline{ m_k}$ and $\alpha^{lai} \beta^{bj \sum_{k = 1}^{s} j_k d(k)} = 1 \, \forall \, b \equiv ax (\text{mod }t)$. We observe $\overline{ m_k} = 1 \Longrightarrow j_k = 0$. Thus the above expression can be rewritten as, 
	\begin{equation*}
	\beta_M = \sum_{i = 0}^L a_i \beta_i + (l \beta_0 + j_1 \beta_1 + \cdots + j_p \beta_p) 
	\end{equation*}
	where $0 \leq j_k < \overline{ n_k}$ and $\alpha^{lai} \beta^{bj \sum_{k = 1}^{p} j_k \text{deg}_Y (P_k)} = 1 \, \forall \, b \equiv ax (\text{mod }t)$.
	Since $L < M$, from (\ref{min gen relation}) we obtain $j_M = 1, a_i = 0 \, \forall \, i = 0 , \cdots , L$ and $j_k = 0 \, \forall \, k \neq M$. Thus $\beta^{bj \text{deg}_Y (P_M)} = 1 \, \forall \, b \equiv ax (\text{mod }t)$. From $2)$ in Section \ref{GS} we have $d(r) = \overline{ m_1} \cdots \overline{ m_{r-1}}$. And, deg$_Y (P_M) = d(\sigma(M)) = \overline{ m_1} \cdots \overline{ m_{\sigma(M)-1}}$. Since $r \geq \sigma(M) \Longrightarrow r-1 \geq \sigma(M)-1$, we thus have deg$_Y (P_M) \mid d(r)$. So, $\beta^{bj d(r)} = 1 \, \forall \, b \equiv ax (\text{mod }t)$. From (\ref{Q_m as eigenfunction}) we then conclude, $Q_r \in A_{i,j,t,x}$. But this contradicts $\gamma_r \notin S^{(A_{i,j,t,x})_{\mathfrak{n}}} (\nu)$. So, $Q_r \in A_{i,j,t,x} \, \forall \, r > \sigma(L) \geq 0$, that is, $Q_r \in A_{i,j,t,x} \, \forall \, r \geq N$ for some $N \in \ZZ_{> 0 }$.
	\par 
	\vspace{ 1 mm} Conversely, we assume $S(\gamma_N, \gamma_{N+1}, \cdots) \, \subset S^{(A_{i,j,t,x})_{\mathfrak{n}}} (\nu)$ for some $N \in \ZZ_{> 0}$. Now $\gamma_i \in \mathbb{Q}_{> 0} \, \forall \, i$ implies $\forall \, i \neq j,  \, \exists \, d_i, d_j \in \mathbb{Z}_{> 0}$ such that $d_i \gamma_i = d_j \gamma_j$. We thus have $d_i \gamma_i = d_{i,N} \gamma_N \, \forall \, i = 0 , \cdots , N-1$. We will now show that, $S^{R_{\mathfrak{m}}} (\nu) = T + S^{(A_{i,j,t,x})_{\mathfrak{n}}} (\nu)$, where $T = \{ \sum_{i=0}^{N-1} \overline{a_i} \gamma_i \mid 0 \leq \overline{a_i} < d_i \}$. Now, $\gamma_i \in S^{R_{\mathfrak{m}}} (\nu) \, \forall \, i = 0, \cdots , N-1 \Longrightarrow T + S^{(A_{i,j,t,x})_{\mathfrak{n}}} (\nu) \subset S^{R_{\mathfrak{m}}} (\nu)$. So it is enough to show $S^{R_{\mathfrak{m}}} (\nu) \subset T + S^{(A_{i,j,t,x})_{\mathfrak{n}}} (\nu)$.
	\begin{align*}
	x \in S^{R_{\mathfrak{m}}} (\nu) &\Longrightarrow x = \sum_{i = 0} ^{N-1} a_i \gamma_i + \sum_{i = N}^l a_i \gamma_i \\
	&\Longrightarrow x = \sum_{i = 0} ^{N-1} \overline{a_i} \gamma_i + \sum_{i = 0}^{N-1} b_i d_i \gamma_i + \sum_{i = N}^l a_i \gamma_i \text{   where } a_i = \overline{a_i} + b_i d_i, \, 0 \leq \overline{a_i} < d_i, \, b_i \in \NN  \\
	&\Longrightarrow x = \sum_{i = 0} ^{N-1} \overline{a_i} \gamma_i + \sum_{i = 0}^{N-1} b_i d_{i,N} \gamma_N + \sum_{i = N}^l a_i \gamma_i \\
	&\Longrightarrow x = \sum_{i = 0} ^{N-1} \overline{a_i} \gamma_i + y , \, \text{ where } y \in S^{(A_{i,j,t,x})_{\mathfrak{n}}} (\nu). \\
	\end{align*}
	Thus we have shown $S^{R_{\mathfrak{m}}} (\nu) \subset T + S^{(A_{i,j,t,x})_{\mathfrak{n}}} (\nu)$. Since $T$ is a finite set, we have $S^{R_{\mathfrak{m}}} (\nu)$ is finitely generated over $S^{(A_{i,j,t,x})_{\mathfrak{n}}} (\nu)$. 
	
	\par  From (\ref{Q_m as eigenfunction}), $(\alpha^{ai}, \beta^{bj}) \cdot Q_N = \beta^{d(N) bj} Q_N \, \forall \, b \equiv ax (\text{mod }t)$. So, $Q_N \in A_{i,j,t,x} \Longleftrightarrow \beta^{d(N) bj} = 1 \, \forall \, b \equiv ax (\text{mod }t)$. Again from $2)$ of Section \ref{GS} we have $d(N) \mid d(M) \, \forall \, M \geq N \geq 1$. Hence we obtain, $Q_N \in A_{i,j,t,x} \Longrightarrow Q_M \in A_{i,j,t,x} \, \forall \, M \geq N \geq 1$. So, $S^{R_{\mathfrak{m}}} (\nu)$ is not finitely generated over $S^{(A_{i,j,t,x})_{\mathfrak{n}}} (\nu)$ if and only if $Q_r \notin A_{i,j,t,x} \, \forall \, r \geq 1$.
	
\end{proof}

\begin{Lemma}\label{not f.g condition}
	Let $i,j,t,x$ be positive integers satisfying the conditions of Theorem \ref{i,j,t,x theorem}. Let assumptions be as in Theorem \ref{invariant subring theorem}. Then $S^{R_{\mathfrak{m}}} (\nu)$ is not finitely generated over $S^{(A_{i,j,t,x})_{\mathfrak{n}}} (\nu)$ if and only if $j \neq n$ and $ \frac{n}{j} \nmid d(l) \, \forall \, l \geq 2$.
\end{Lemma}

\begin{proof}
	Suppose that $S^{R_{\mathfrak{m}}} (\nu)$ is not finitely generated over $S^{(A_{i,j,t,x})_{\mathfrak{n}}} (\nu)$. Then $Q_l \notin A_{i,j,t,x} \, \forall \, l \geq 1$. From (\ref{Q_m as eigenfunction}), if $j = n$, then $(\alpha^{ai}, \beta^{bn}) \cdot Q_l = \beta^{d(l)bn} Q_l = Q_l$, that is $Q_l \in A_{i,n,t,x}$, which is a contradiction. So $j \neq n$. And, for some $l \geq 2$, $\frac{n}{j} \mid d(l) \Longrightarrow n \mid d(l)j$. Then, $(\alpha^{ai}, \beta^{bj}) \cdot Q_l = \beta^{d(l)bj} Q_l  = Q_l$, that is $Q_l \in A_{i,j,t,x}$, which is again a contradiction. So, $\frac{n}{j} \nmid d(l) \, \forall \, l \geq 2$.
	\par  Conversely, suppose $j \neq n$ and $\frac{n}{j} \nmid d(l) \, \forall \, l \geq 2$, that is, $\frac{n}{j} \nmid d(l) \, \forall \, l \geq 1$. Now, $(x,t) = 1 \Longrightarrow ax \equiv 1 (\text{mod }t )$ for some $a \in \ZZ$, so, $(\alpha^{ai}, \beta^{j}) \in H_{i,j,t,x}$. From (\ref{Q_m as eigenfunction}), $(\alpha^{ai}, \beta^{j}) \cdot Q_l = \beta^{d(l)j} Q_l \neq Q_l$ for all $l \geq 1$, as $n \nmid d(l)j$. So we have $Q_l \notin A_{i,j,t,x} \, \forall \, l \geq 1$. Hence $S^{R_{\mathfrak{m}}} (\nu)$ is not finitely generated over $S^{(A_{i,j,t,x})_{\mathfrak{n}}} (\nu)$.
\end{proof}

\begin{Proposition}\label{gcd > t}
	Let $i,j,t,x$ be positive integers satisfying the conditions of Theorem \ref{i,j,t,x theorem}, such that $( \frac{m}{i}, \frac{n}{j} ) > t \geq 1$. Suppose that $\nu$ is a rational rank $1$ non discrete valuation dominating $R_{\mathfrak{m}}$, with a generating sequence (\ref{Gen Seq}) $\{Q_l \}_{l \geq 0}$, where $Q_0 = X, Q_1 = Y$ as in Section \ref{GS}. Then $\{Q_l \}_{l \geq 0}$ is not a sequence of eigenfunctions for $H_{i,j,t,x}$. 
\end{Proposition}

\begin{proof}
	Let $d = ( \frac{m}{i}, \frac{n}{j})$. Then $1 \leq t < d \leq  \text{min }\{ \frac{m}{i}, \frac{n}{j} \}$. So, $t < \frac{m}{i} \text{ and } t < \frac{n}{j}$. We recall, $H_{i,j,t,x} = \{ ( \alpha^{ai}, \beta^{bj} ) \, | \, b \equiv ax (\text{mod } t) \}$. Thus $(\alpha^{ti}, 1) , \, (1, \beta^{tj}) \in H_{i,j,t,x}$. Let $\{ Q_l\}_{l\geq 0}$ be the generating sequence (\ref{Gen Seq}) with $Q_0 = X, \, Q_1 = Y$. Let $\nu (Q_l) = \gamma_l \, \forall \, l \geq 0$. By Equation $(8)$ in [\ref{CV1}], $Q_2 = Y^s - \lambda X^r$, where $\lambda \in K \setminus \{ 0 \}$, $s \gamma_1 = r \gamma_0$, and $s$ = min $\{q \in \mathbb{Z}_{> 0} \mid q \gamma_1 \in \gamma_0 \mathbb{Z} \}$. From (\ref{K-algebra action}), we have,
	\begin{align*}
	(\alpha^{ti}, 1) \cdot Q_2 &= (\alpha^{ti}, 1) \cdot [Y^s - \lambda X^r] = Y^s - \lambda \alpha^{rti} X^r .\\
	(1, \beta^{tj}) \cdot Q_2 &=  (1, \beta^{tj}) \cdot [Y^s - \lambda X^r]  =  \beta^{stj} Y^s - \lambda X^r .\\
	\end{align*} 
	If $Q_2$ was an eigenfunction of $H_{i,j,t,x}$, then $m \mid rti \Longrightarrow r = r_1 \frac{m}{ti}$, where $r_1 \in \mathbb{Z}_{> 0}$. Similarly, $n \mid stj  \Longrightarrow s = s_1 \frac{n}{tj}$, where $s_1 \in \mathbb{Z}_{> 0}$. And, $s \gamma_1 = r \gamma_0 \Longrightarrow s_1 \frac{n}{tj} \gamma_1 = r_1 \frac{m}{ti} \gamma_0$. So, $s_1 \frac{n}{dj} \gamma_1 = r_1 \frac{m}{di} \gamma_0$. Now, $d \, | \, \frac{n}{j}$ implies $ s_1 \frac{n}{dj} \in \mathbb{Z}_{> 0}$. Similarly, $r_1 \frac{m}{di} \in \mathbb{Z}_{> 0}$. Thus, $s_1 \frac{n}{dj} \gamma_1 \in \gamma_0 \mathbb{Z}$. But $t < d$ implies $s_1 \frac{n}{dj} < s_1 \frac{n}{tj} = s$, and this contradicts the minimality of $s$. Thus $Q_2$ is not an eigenfunction of $H_{i,j,t,x}$. So, $\{ Q_l\}_{l \geq 0}$ is not a generating sequence of eigenfunctions for $H_{i,j,t,x}$.		
\end{proof}

We know, if $\omega$ is a primitive $l$-th root of unity in $K$, then $\{ \omega^k \mid 1 \leq k \leq l \}$ is a complete list of all $l$-th roots of unity in $K$, and $\{ \omega^k \mid 1 \leq k \leq l \text{ and } (k,l) = 1 \} $ is a complete list of all primitive $l$-th roots of unity in $K$. 
\newline We have, $\alpha$ is a primitive $m$-th root of unity and $\beta$ is a primitive $n$-th root of unity in K. Let $\delta$ be a primitive $mn$-th root of unity in $K$. Then $\delta^{n}$ is a primitive $m$-th root of unity. Now, $S_{\alpha} = \{ \alpha^k \mid 1 \leq k \leq m \text{ and } (k,m) = 1 \}$ is a complete list of all primitive $m$-th roots of unity in $K$. And, $S_{\delta^n} = \{ \delta^{k n} \mid 1 \leq k \leq m \text{ and } (k,m) = 1 \}$ is also a complete list of all primitive $m$-th roots of unity. Thus, $\alpha = \delta^{w_1 n} \text{ where } (w_1, m) = 1 \text{ and } 1 \leq w_1 \leq m$. Similarly, $\beta = \delta^{w_2 m} \text{ where } (w_2, n) = 1 \text{ and } 1 \leq w_2 \leq n$.

\begin{Remark}\label{relation between alpha and beta}
	Let $p,q \in \mathbb{Z}$. With the notation introduced above, $\beta^p = \alpha^q \Longleftrightarrow \frac{p w_2}{n} - \frac{q w_1}{m} \in \mathbb{Z}$.
\end{Remark}

\begin{proof}
	We have, $\beta = \delta^{w_2 m}$ and $\alpha = \delta^{w_1 n}$, where $\delta$ is a primitive $mn$-th root of unity. 
	\newline Thus, $\beta^p = \alpha^q \Longleftrightarrow \delta^{w_2 m p} = \delta^{w_1 n q} \Longleftrightarrow mn \mid (w_2 mp - w_1 nq) \Longleftrightarrow \frac{p w_2}{n} - \frac{q w_1}{m} \in \mathbb{Z}$.
\end{proof}

\begin{Proposition}\label{p nmid N}
	Let $i,j,t,x$ be positive integers satisfying the conditions of Theorem \ref{i,j,t,x theorem}, such that $ (\frac{m}{i}, \frac{n}{j} ) = t, \, t > 1$. Set $\frac{m}{i} = Mt$, and $\frac{n}{j} = Nt$, where $M,N \in \mathbb{Z}_{> 0}$ and $(M,N) = 1$. Suppose that $\exists$ a prime number $p$ such that $p \mid t$ but $p \nmid N$. Suppose that $\nu$ is a rational rank $1$ non discrete valuation dominating $R_{\mathfrak{m}}$ with a generating sequence (\ref{Gen Seq}) of eigenfunctions for $H_{i,j,t,x}$. Then $S^{R_{\mathfrak{m}}} (\nu)$ is not finitely generated over $S^{(A_{i,j,t,x})_{\mathfrak{n}}} (\nu)$.
\end{Proposition}

\begin{proof}
	Let $\{ Q_l \}_{l \geq 0}$ be the generating sequence (\ref{Gen Seq}) of the valuation $\nu$, where $Q_0 = X, Q_1 = Y$, and each $Q_l$ is an eigenfunction for $H_{i,j,t,x}$. Let $\gamma_l = \nu (Q_l) \, \forall \, l \geq 0$. Without any loss of generality, we can assume $\gamma_0 = 1$. Since $\nu$ is a rational valuation, we can write $\gamma_k = \frac{a_k}{b_k} \, \forall \, k \geq 1$, where $(a_k, b_k) = 1$. We have, $p \mid t$, and $p \nmid N$ for a prime $p$. So $(p,N) = 1$. So $\exists N_1 \in \mathbb{Z}$ such that $N N_1 \equiv 1 (\text{mod }p)$. Let $w_1$ and $w_2$ be as in Remark \ref{relation between alpha and beta}. Now $(m, w_1) = 1$ and $t \mid m$. So $(t, w_1) = 1$. So $(p, w_1) = 1$. So $\exists \, \overline{w_1} \in \mathbb{Z}$ such that $w_1 \overline{w_1} \equiv 1 (\text{mod }p)$.  
	\newline We now use induction to show the following $\forall \, k \geq 1$,
	\begin{equation}\label{induction for p nmid N}
	\begin{split} 
	(p , \overline{m_k}) &= 1, \, (p, b_k) = 1 \\
	a_k &\equiv b_k M N_1 x w_2 \overline{w_1} d(k) \,  (\text{mod }p) .\\
	\end{split}
	\end{equation}
	We have $\gamma_1 = \frac{a_1}{b_1}$, where $(a_1, b_1) = 1$. So $\overline{ m_1} = b_1$. By Equation $(8)$ in [\ref{CV1}], we have $Q_2 = Y^{b_1} - \lambda_1 X^{a_1}$, for some $\lambda_1 \in K \setminus \{0\}$. Recall, $H_{i,j,t,x} = \{ (\alpha^{ai}, \beta^{bj}) \, | \, b \equiv ax (\text{mod }t) \}$. So $(\alpha^i, \beta^{xj}) \in H_{i,j,t,x}$. Now, $(\alpha^i, \beta^{xj}) \cdot Q_2 = \beta^{b_1 xj} Y^{b_1} - \lambda_1 \alpha^{a_1 i} X^{a_1}$. Since $Q_2$ is an eigenfunction for $H_{i,j,t,x}$, we have 
	\begin{align*}
	\beta^{b_1 xj} = \alpha^{a_1 i} &\Longrightarrow \frac{b_1 xj w_2}{n} - \frac{a_1 i w_1}{m} \in \mathbb{Z} \text{ by Remark }\ref{relation between alpha and beta}\\
	&\Longrightarrow \frac{b_1 x w_2}{Nt} - \frac{a_1 w_1}{Mt} \in \mathbb{Z} \\
	&\Longrightarrow MNt \mid [b_1 xM w_2 - a_1 N w_1] \\
	&\Longrightarrow b_1 M N_1 x w_2 \overline{w_1} \equiv a_1  (\text{mod }p) \text{ as } p \mid t.\\
	\end{align*}
	If $(p,b_1) \neq 1$, then $p \mid b_1 \Longrightarrow p \mid a_1$. But this contradicts $(a_1, b_1) = 1$. So, $(p, b_1) = 1$. Since $\overline{ m_1} = b_1$, we thus have $(p, \overline{ m_1}) = 1$. Thus we have the induction step for $k = 1$. 
	\newline Suppose (\ref{induction for p nmid N}) is true for $k = 1, \cdots , l-1$. From (\ref{Q_m as eigenfunction}) we have $(\alpha^{ai}, \beta^{bj}) \cdot Q_k = \beta^{ d(k) bj} Q_k \, \forall \, k \geq 1$, $\forall \, (\alpha^{ai}, \beta^{bj}) \in H_{i,j,t,x}$. By Equation $(8)$ in [\ref{CV1}] we have, $Q_{l+1} = Q_l ^{\overline{ m_l}} - \lambda_l X^{c_0} Y^{c_1} Q_2 ^{c_2} \cdots Q_{l-1} ^{c_{l-1}}$ where $\lambda_l \in K \setminus\{ 0 \}, \, 0 \leq c_k < \overline{ m_k} \, \forall \, k = 1, \cdots , l-1$ and $\overline{ m_l} \gamma_l = \sum_{k=0}^{l-1} c_k \gamma_k$.
	\newline $(\alpha^i, \beta^{xj}) \cdot Q_{l+1} = \beta^{xj \overline{ m_l} d(l)} Q_l ^{\overline{ m_l}} - \lambda_l \alpha^{i c_0} \beta^{xj [\sum_{k=1}^{l-1} c_k d(k)]} X^{c_0} Y^{c_1} Q_2 ^{c_2} \cdots Q_{l-1} ^{c_{l-1}}$. Since $Q_{l+1}$ is an eigenfunction for $H_{i,j,t,x}$, we have
	\begin{align*}
	&\beta^{xj \overline{ m_l} d(l)} = \alpha^{i c_0} \beta^{xj [\sum_{k=1}^{l-1} c_k d(k)]} \\
	\Longrightarrow &\beta^{xj[\overline{ m_l} d(l) - \sum_{k=1}^{l-1} c_k d(k)]} = \alpha^{i c_0} \\
	\Longrightarrow & \frac{x[\overline{ m_l} d(l) - \sum_{k=1}^{l-1} c_k d(k)] w_2}{Nt} - \frac{c_0 w_1}{Mt} \in \mathbb{Z} \text{ by Remark }\ref{relation between alpha and beta}\\
	\Longrightarrow &MNt \mid \mathlarger{[}Mxw_2 \overline{ m_l} d(l) - M x w_2 \sum_{k=1}^{l-1} c_k d(k) - N c_0 w_1 \mathlarger{]}\\
	\Longrightarrow &p \mid \mathlarger{[} Mx w_2 \overline{ m_l} d(l) - M x  w_2 \sum_{k=1}^{l-1} c_k d(k) - N c_0 w_1 \mathlarger{]} \\
	\Longrightarrow &M N_1 x w_2 \overline{w_1} \, \overline{ m_l} d(l) \equiv \mathlarger{[} M N_1 x w_2 \overline{w_1} \sum_{k=1}^{l-1} c_k d(k) + c_0 \mathlarger{]} (\text{mod }p) .\\
	\end{align*}	
	Now, $p \mid \overline{ m_l}  \Longrightarrow c_0 = \lambda p - M N_1 x w_2 \overline{w_1} \sum_{k=1}^{l-1} c_k d(k)$, where $\lambda \in \mathbb{Z}$. Let $\overline{ m_l} = p M_l$, where $M_l \in \mathbb{Z}_{> 0}$. So, $\overline{ m_l} \gamma_l = p M_l \gamma_1 = c_0 + \sum_{k=1}^{l-1} c_k \gamma_k = \lambda p + \sum_{k=1}^{l-1} c_k [ \gamma_k - M N_1 x w_2 \overline{w_1} d(k) ] $.	
	\newline By our induction statement, $\forall \, k = 1 , \cdots , l-1$, we have $a_k = t_k p + b_k M N_1 x w_2 \overline{w_1} d(k)$, where $t_k \in \mathbb{Z}$. Thus,
	\begin{equation*}
	p M_l \gamma_l = \lambda p + \sum_{k=1}^{l-1} c_k [ \frac{t_k p + b_k M N_1 x w_2 \overline{w_1} d(k)}{b_k} - M N_1 x w_2 \overline{w_1} d(k) ] =  \lambda p +  p \sum_{k=1}^{l-1} c_k t_k \frac{1}{b_k}.
	\end{equation*} 
	Now $(a_k, b_k) = 1 \Longrightarrow \exists \, h_k \in \ZZ$ such that $h_k a_k \equiv 1 (\text{mod }b_k)$. Let $h_k a_k - 1 = \zeta_k b_k$, where $\zeta_k \in \mathbb{Z}$.
	So, $\frac{1}{b_k} = \frac{h_k a_k - (h_k a_k -1)}{b_k} = h_k \gamma_k - \zeta_k$. Then, $p M_l \gamma_l = \lambda p +  p \sum_{k=1}^{l-1} c_k t_k [h_k \gamma_k - \zeta_k]$ implies
	\begin{equation*}
	M_l \gamma_l = \lambda  +  \sum_{k=1}^{l-1} c_k t_k [h_k \gamma_k - \zeta_k] \in G(\gamma_0, \cdots , \gamma_{l-1}) .
	\end{equation*}
	But this contradicts the minimality of $\overline{ m_l}$. So $p \nmid \overline{ m_l}$. So $(p, \overline{ m_l}) = 1$.
	\par Now, $\overline{ m_l} \gamma_l =  c_0 + \sum_{k=1}^{l-1} c_k \gamma_k \Longrightarrow \overline{ m_l} \frac{a_l}{b_l} = c_0 + \sum_{k=1}^{l-1} c_k \frac{a_k}{b_k} \Longrightarrow \overline{ m_l} a_l \prod_{k=1}^{l-1} b_k = c_0 B  + B \sum_{k=1}^{l-1} c_k \frac{a_k}{b_k}$, where $B = \prod_{k=1}^{l} b_k$. From the induction hypothesis, $\frac{a_k}{b_k} B = [t_k p + b_k M N_1 x w_2 \overline{w_1} d(k)] \frac{B}{b_k}$. So, 
	\begin{align*}
	\overline{ m_l} a_l \prod_{k=1}^{l-1} b_k & = c_0 B + \sum_{k=1}^{l-1} c_k [t_k p + b_k M N_1 x w_2 \overline{w_1} d(k)] \frac{B}{b_k} \\
	\Longrightarrow	\overline{ m_l} a_l \prod_{k=1}^{l-1} b_k & \equiv [c_0  + M N_1 x w_2 \overline{w_1} \sum_{k=1}^{l-1} c_k d(k) ] B (\text{mod }p) .\\
	\end{align*}
	Since, $M N_1 x w_2 \overline{w_1} \, \overline{ m_l} d(l) \equiv [ M N_1 x w_2 \overline{w_1} \sum_{k=1}^{l-1} c_k d(k) + c_0 ] (\text{mod }p)$, we have
	\begin{align*}
	\overline{ m_l} a_l \prod_{k=1}^{l-1} b_k  \equiv M N_1 x w_2 \overline{w_1} \, \overline{ m_l} d(l) \prod_{k = 1}^{l} b_k (\text{mod }p) .
	\end{align*}
	Since $(p, \overline{ m_l}) = 1, \, (p, b_k) = 1 \, \forall \, k = 1, \cdots , l-1$, we have $a_l  \equiv M N_1 x w_2 \overline{w_1} d(l) b_l \, (\text{mod }p)$. If $p \mid b_l$, then $p \mid a_l$ which contradicts $(a_l, b_l) = 1$. So $(p, b_l) = 1$. Thus we have the induction step for $k = l$.
	\par In particular, by induction we have $(p, \overline{ m_k}) = 1 \, \forall \, k \geq 1$. Since $d(k) = \overline{ m_1} \cdots \overline{ m_{k-1}}$ (by $2)$, Section \ref{GS}), we have $(p, d(k)) = 1 \, \forall \, k \geq 2$. So $p \nmid d(k) \, \forall \, k \geq 2 \Longrightarrow t \nmid d(k) \, \forall \, k \geq 2 \Longrightarrow \frac{n}{j} = Nt  \nmid d(k) \, \forall \, k \geq 2$. Thus by Lemma \ref{not f.g condition}, we have $S^{R_{\mathfrak{m}}} (\nu)$ is not finitely generated over $S^{(A_{i,j,t,x})_{\mathfrak{n}}} (\nu)$.
\end{proof}

\begin{Proposition}\label{p divides N}
	Let $i,j,t,x$ be positive integers satisfying the conditions of Theorem \ref{i,j,t,x theorem}, such that $(\frac{m}{i}, \frac{n}{j}) = t$ and $t > 1$. Set $\frac{m}{i} = Mt$ and $\frac{n}{j} = Nt$ where $M,N \in \mathbb{Z}_{> 0} $ and $(M,N) = 1$. Suppose that for any prime number $p$ which divides $t$, the number $p$ also divides $N$. Suppose that $\nu$ is a rational rank $1$ non discrete valuation dominating $R_{\mathfrak{m}}$ with a generating sequence (\ref{Gen Seq}) of eigenfunctions for $H_{i,j,t,x}$. Then $S^{R_{\mathfrak{m}}} (\nu)$ is not finitely generated over $S^{(A_{i,j,t,x})_{\mathfrak{n}}} (\nu)$.
\end{Proposition}

\begin{proof}
	Since $(x,t) = 1, \, \exists \, r \in \mathbb{Z}_{> 0}$ such that $rx \equiv 1 (\text{mod }t)$. So $(r,t) = 1$. Recall, $\alpha = \delta^{w_1 n},  \beta = \delta^{w_2 m}$, where $\delta$ is a primitive $mn$-th root of unity, and $(w_1,m) = 1, (w_2,n) = 1, 1 \leq w_1 \leq m$ and $1 \leq w_2 \leq n$. Now, $M \mid m \Longrightarrow (w_1, M) = 1$. Similarly, $(w_2, N) = 1$, $(w_1, t) = 1$, $(w_2, t) = 1$. So $\exists \, \overline{ w_1}, \, \overline{ w_2} \in \mathbb{Z}_{> 0}$ such that $w_1 \overline{ w_1} \equiv 1 (\text{mod }t)$ and $w_2 \overline{ w_2} \equiv 1 (\text{mod }t)$.
	\newline Write $N = \overline{N} N^{\prime}$, where $\overline{N}$ is the largest factor of $N$ such that $(\overline{N}, x) = 1$. If $\overline{N} = 1$, then for any prime $p$ dividing $N$, we have $p \mid x$. So in particular $p \mid t \Longrightarrow p \mid x$. But this is a contradiction as $(t,x) = 1$. So $\overline{N} > 1$ if $N > 1$. We will now show $(\overline{N}, N^{\prime}) = 1$. Suppose the contrary. Then $\exists$ a prime $p$ such that $p \mid \overline{N}$ and $p \mid N^{\prime}$. $p \mid \overline{N} \Longrightarrow (p ,x) = 1 \Longrightarrow (\overline{N}p , x) = 1$. And, $\overline{N}N^{\prime} = N \Longrightarrow p\overline{N} \mid N$. This contradicts the maximality of $\overline{N}$. So $(\overline{N}, N^{\prime}) = 1$. Hence $(N,x) = (N^{\prime}, x)$. We will now show that $(t, N^{\prime}) = 1$. Suppose $\exists$ a prime $p$ such that $p \mid t$ and $p \mid N^{\prime}$. Then $p \mid t, p \mid N$ and $p \nmid \overline{N}$. Thus $p \mid t$ and $p \mid x$, which is a contradiction as $t$ and $x$ are coprime. Thus $(t, N^{\prime}) = 1$. Also $(N, w_2) = 1$ implies $(\overline{N}, w_2) = 1$.
	\newline Let $\{ Q_l \}_{l \geq 0}$ be the generating sequence (\ref{Gen Seq}) of the valuation $\nu$, where $Q_0 = X, Q_1 = Y,$ and each $Q_l$ is an eigenfunction for $H_{i,j,t,x}$. Let $\gamma_l = \nu (Q_l) \, \forall \, l \geq 0$. Without any loss of generality, we can assume $\gamma_0 = 1$. Let $\gamma_1 = \frac{a_1}{b_1}$, where $(a_1, b_1) = 1$. So $\overline{ m_1} = b_1$. By Equation $(8)$ in [\ref{CV1}], we have $Q_2 = Y^{b_1} - \zeta_1 X^{a_1}$ for some $\zeta_1 \in K \setminus \{ 0 \}$. Now,$(\alpha^i, \beta^{xj}) \in H_{i,j,t,x}$. By (\ref{Q_m as eigenfunction}), $(\alpha^{ai}, \beta^{bj}) \cdot Q_k = \beta^{ d(k) bj} Q_k \, \forall \, k \geq 1, \, \forall \, (\alpha^{ai}, \beta^{bj}) \in H_{i,j,t,x}$. Now, $(\alpha^i, \beta^{xj}) \cdot Q_2 = (\alpha^i, \beta^{xj}) \cdot [Y^{b_1} - \zeta_1 X^{a_1}] = \beta^{b_1 xj} Y^{b_1} - \zeta_1 \alpha^{a_1 i} X^{a_1}$. Since $Q_2$ is an eigenfunction for $H_{i,j,t,x}$, we have
	\begin{align*}
	\beta^{b_1 xj} = \alpha^{a_1 i} &\Longrightarrow \frac{b_1 xw_2}{Nt} - \frac{a_1w_1}{Mt} \in \mathbb{Z} \text{ by Remark }\ref{relation between alpha and beta}\\
	&\Longrightarrow M\overline{N} t \mid [M b_1 x w_2 - N a_1 w_1]\\
	&\Longrightarrow M \mid a_1 \text{ and } \overline{N} \mid b_1   \text{ as } (\overline{N},w_2) = 1,\, (M,w_1)= 1, \, (M,N) = 1, \, (\overline{N}, x) = 1.\\
	\end{align*}  
	Let $a_1 = M a_1^{\prime}$ and $b_1 = \overline{N} b_1^{\prime}$. Then, $M\overline{N}t \mid [M\overline{N}b_1^{\prime}x w_2 - NMa_1^{\prime}w_1]$ implies $b_1^{\prime} \equiv r a_1^{\prime} w_1 \overline{ w_2} N^{\prime} (\text{mod }t)$ as $rx \equiv 1 (\text{mod }t)$ and $N = \overline{N} N^{\prime}$. Now, $\gamma_1 = \frac{a_1}{b_1} = \frac{M a_1^{\prime}}{\overline{N} b_1^{\prime}}$. $(a_1, b_1) = 1 \Longrightarrow (\overline{N}, a_1^{\prime}) = 1, \, (a_1^{\prime}, b_1^{\prime}) = 1$ and $(M, b_1^{\prime}) = 1$. Rename $a_1^{\prime} = u$ and $b_1^{\prime} = r^{\prime}$. Then $(u, \overline{N}) = 1$. If $(u,t) \neq 1$, then $\exists$ a prime $p$ such that $p \mid t$ and $p \mid u$. Thus $p \mid t$, $p \mid N$ and $p \nmid \overline{N}$, since for any prime $p$ dividing $t$, $p$ also divides $N$. So $p \mid t$ and $p \mid N^{\prime}$. But we have established earlier that $(t, N^{\prime}) = 1$. So $(u,t) = 1$. And, $r^{\prime} \equiv ruw_1 \overline{ w_2} N^{\prime} (\text{mod }t) \Longrightarrow r^{\prime}x \equiv uw_1 \overline{ w_2} N^{\prime} (\text{mod }t)$. Thus, 
	\begin{align}\label{expression of gamma_1}
	\gamma_1 = \frac{Mu}{\overline{N} r^{\prime}} \text{ where } (u, \overline{N}) = 1, (u,t) = 1, (u, r^{\prime}) = 1, (M, r^{\prime}) = 1, r^{\prime} \equiv ruw_1 \overline{ w_2} N^{\prime} (\text{mod }t).
	\end{align}
	We will now use induction to show that $\forall \, k \geq 2$,
	\begin{equation}\label{induction for p divides N}
	\begin{split}
	& \gamma_k = Mu \overline{ m_2} \cdots \overline{ m_{k-1}} + \frac{M\overline{N}t\lambda_k}{\overline{ m_1} \cdots \overline{ m_k} } \text{ for some } \lambda_k \in \mathbb{Z} \\
	& (t, \overline{ m_k}) = 1 .\\ 
	\end{split}
	\end{equation}
	By Equation $(8)$ in [\ref{CV1}] we have, $Q_3 = Q_2^{\overline{ m_2}} - \zeta_2 X^{c_0} Y^{c_1}$ where $\zeta_2 \in K \backslash \{ 0 \}, \,  c_0 \in \mathbb{Z}_{> 0}, \, 0 \leq c_1 < \overline{ m_1}$. $(\alpha^i, \beta^{xj}) \cdot Q_3 = \beta^{xj \overline{ m_2} \, \overline{ m_1}} Q_2^{\overline{ m_2}} - \zeta_2 \alpha^{i c_0 } \beta^{xj c_1} X^{c_0} Y^{c_1}$. Since $Q_3$ is an eigenfunction for $H_{i,j,t,x}$, we have
	\begin{align*}
	\beta^{xj \overline{ m_2} \, \overline{ m_1}} = \alpha^{i c_0 } \beta^{xj c_1} &\Longrightarrow \beta^{xj [\overline{ m_2} \, \overline{ m_1} - c_1]} = \alpha^{i c_0} \\
	&\Longrightarrow \frac{x [\overline{ m_2} \, \overline{ m_1} - c_1] w_2}{Nt} - \frac{c_0 w_1}{Mt} \in \mathbb{Z} \text{ by Remark } \ref{relation between alpha and beta}\\
	&\Longrightarrow M\overline{N}t \mid [M \overline{N}  r^{\prime} x w_2 \overline{ m_2} - Mx w_2 c_1 - N c_0 w_1]  \text{   as } \overline{ m_1} = \overline{N} r^{\prime}\\
	&\Longrightarrow M \mid c_0 \text{ and } \overline{N} \mid c_1  \text{  as } (M,N) = 1, \, (M, w_1) = 1, \, (\overline{N},w_2) = 1, \,  (\overline{N}, x) = 1. \\
	\end{align*}
	Let $c_0 = M c_0^{\prime}$ and $c_1 = \overline{N} c_1^{\prime}$. Plugging them in the above expression and using (\ref{expression of gamma_1}), we obtain,
	\begin{align*}
	& M \overline{N} t \mid [M \overline{N} r^{\prime} x w_2 \overline{ m_2} - M  x w_2 \overline{N} c_1^{\prime} - NMc_0^{\prime} w_1 ] \\
	&\Longrightarrow r^{\prime} x w_2 \overline{ m_2} \equiv [w_1 c_0^{\prime} N^{\prime} + x w_2 c_1^{\prime}] (\text{mod }t) \\
	&\Longrightarrow u w_1 \overline{ m_2} N^{\prime} \equiv [w_1 c_0^{\prime} N^{\prime} + x w_2 c_1^{\prime}] (\text{mod }t) \\
	&\Longrightarrow r^{\prime} u \overline{ m_2} \equiv [r^{\prime} c_0^{\prime} + u c_1^{\prime}] (\text{mod }t). \\
	\end{align*}
	So, $\overline{ m_2} \gamma_2 = c_0 + c_1 \gamma_1 = M c_0^{\prime} + \overline{N} c_1^{\prime} \frac{Mu}{\overline{N} r^{\prime}} = M [\frac{c_0^{\prime} r^{\prime} +  c_1^{\prime} u}{r^{\prime}}] = M[\frac{r^{\prime} u \overline{ m_2} + \lambda_2 t}{r^{\prime}}] = Mu \overline{ m_2} + \frac{M \overline{N} t \lambda_2 }{\overline{ m_1}} $ for some $\lambda_2 \in \ZZ$. Thus, $\gamma_2 = Mu + \frac{M\overline{N}t \lambda_2}{\overline{ m_1} \, \overline{ m_2}}$.
	\newline We will now show $(t, \overline{ m_2}) = 1$. Suppose if possible $\exists$ a prime $p$ such that $p \mid t$ and $p \mid \overline{ m_2}$. Let $\overline{ m_2} = p M_2$. So, $\gamma_2 = Mu + \frac{M\overline{N}t \lambda_2}{\overline{ m_1} \, \overline{ m_2}} \Longrightarrow \overline{ m_2} \gamma_2 = Mu\overline{ m_2} + \frac{M\overline{N}t\lambda_2}{\overline{ m_1}} \Longrightarrow pM_2 \gamma_2 = pMuM_2 + \frac{Mt\lambda_2}{r^{\prime}} \Longrightarrow r^{\prime} M_2 \gamma_2 = r^{\prime} M u M_2 + M \lambda_2 \frac{t}{p}$.
	\newline $(w_1, t) = 1$. $(N^{\prime}, t) = 1$. $rx \equiv 1 (\text{mod }t)$ implies $(r,t) = 1$. $w_2 \overline{ w_2} \equiv 1 (\text{mod }t)$ implies $(\overline{ w_2},t) = 1$. And, $(u,t) = 1$ by (\ref{expression of gamma_1}). So, $r^{\prime} \equiv ruw_1 \overline{ w_2} N^{\prime} (\text{mod }t) \Longrightarrow (r^{\prime}, t) = 1$. So $\exists \, r_1 \in \mathbb{Z}$ such that $r_1 r^{\prime} \equiv 1 (\text{mod }t)$. So in particular, $r_1 r^{\prime} \equiv 1 (\text{mod }p) \, \forall$ prime $p$ dividing $t$. We then have,
	\begin{align*}
	&r_1 r^{\prime} M_2 \gamma_2 = r_1 r^{\prime} MuM_2 + r_1 M \lambda_2 \frac{t}{p} \\
	&\Longrightarrow (1 + \mu_2 p) M_2 \gamma_2 = r_1 r^{\prime} MuM_2  + r_1 M \lambda_2 \frac{t}{p} \text{ for some } \mu_2 \in \mathbb{Z} \\
	&\Longrightarrow M_2 \gamma_2 + \mu_2 \overline{ m_2} \gamma_2 \in \mathbb{Z} \subset G(\gamma_0, \gamma_1) \Longrightarrow M_2 \gamma_2 \in G(\gamma_0, \gamma_1) .\\
	\end{align*}
	But this contradicts the minimality of $\overline{ m_2}$. So for any prime $p$ dividing $t$, we have $p \nmid \overline{ m_2}$. Thus $(t, \overline{ m_2}) = 1$. We now have the induction step for $k = 2$.  	
	\newline Suppose (\ref{induction for p divides N}) is true for $k = 3, \cdots , l-1$. By Equation $(8)$ in [\ref{CV1}] we have,
	$Q_{l+1} = Q_l^{\overline{ m_l}} - \zeta_l X^{c_0} Y^{c_1} Q_2^{c_2} \cdots Q_{l-1}^{c_{l-1}}$ where $\zeta_l \in K \setminus \{0 \}, \, c_0 \in \mathbb{Z}_{> 0}, \, 0 \leq c_k < \overline{ m_k} \, \forall \, k = 1, \cdots , l-1$ and $\overline{ m_l} \gamma_l = \sum_{k=0}^{l-1} c_k \gamma_k$. By $2)$ of Section \ref{GS} we have $d(l) = \prod_{k =1}^{l-1} \overline{ m_k} \, \forall \, l \geq 2$. Again, $\overline{ m_1} = \overline{N} r^{\prime}$ by (\ref{expression of gamma_1}). So $\forall \, l \geq 2$, $d(l) = \overline{N} r^{\prime} \overline{d(l)}$, where $\overline{d(l)} = \frac{d(l)}{\overline{ m_1}}$. Thus, $ \forall \, l \geq 3, \overline{d(l)} = \prod_{k =2}^{l-1} \overline{ m_k}$. 
	\newline Now, $(\alpha^i, \beta^{xj}) \cdot Q_{l+1} = \beta^{xj \overline{ m_l} d(l)} Q_l^{\overline{ m_l}} - \zeta_l \alpha^{i c_0} \beta^{xj [\sum_{k=1}^{l-1} c_k d(k)]} X^{c_0} Y^{c_1} Q_2^{c_2} \cdots Q_{l-1}^{c_{l-1}}$. Since $Q_{l+1}$ is an eigenfunction for $H_{i,j,t,x}$ we have
	\begin{align*}
	&\Longrightarrow \beta^{xj [ d(l+1) -  \sum_{k=1}^{l-1} c_k d(k)] } = \alpha^{i c_0} \\
	&\Longrightarrow \frac{x w_2[ d(l+1) -  \sum_{k=1}^{l-1} c_k d(k)]}{Nt} - \frac{c_0w_1}{Mt} \in \mathbb{Z} \text{ by Remark } \ref{relation between alpha and beta}\\
	&\Longrightarrow M\overline{N}t \mid [Mxw_2 \overline{N}r^{\prime}\overline{d(l+1)} - Mxw_2c_1 - Mxw_2\overline{N}r^{\prime} \sum_{k=2}^{l-1} c_k \overline{d(k)} - Nc_0w_1 ]  \\
	&\Longrightarrow M \mid c_0 \text{ and } \overline{N} \mid c_1  \text{ as } (M,N) = 1, \, (M,w_1)= 1, \, (\overline{N}, x)= 1,\, (\overline{N},w_2) = 1 .\\
	\end{align*}
	Let $c_0 = M c_0^{\prime}$ and $c_1 = \overline{N} c_1^{\prime}$. Plugging them in the above expression, and using (\ref{expression of gamma_1}), we obtain
	\begin{align*}
	&M \overline{N} t \mid [Mxw_2 \overline{N} r^{\prime}\overline{d(l+1)} - Mxw_2 \overline{N} c_1^{\prime} - Mxw_2 \overline{N} r^{\prime} \sum_{k=2}^{l-1} c_k \overline{d(k)} - NMw_1c_0^{\prime} ] \\
	&\Longrightarrow t \mid [xw_2 r^{\prime}\overline{d(l+1)} - xw_2c_1^{\prime} - x w_2r^{\prime}\sum_{k=2}^{l-1} c_k \overline{d(k)} - w_1 c_0^{\prime} N^{\prime}   ] \\
	&\Longrightarrow r^{\prime} xw_2 \overline{d(l+1)} \equiv [c_0^{\prime}w_1 N^{\prime} + c_1^{\prime} x w_2 + r^{\prime}x w_2 \sum_{k=2}^{l-1} c_k \overline{d(k)}] (\text{mod }t) \\
	&\Longrightarrow  r^{\prime}u  \overline{d(l+1)} \equiv [ r^{\prime} c_0^{\prime} + c_1^{\prime} u   + r^{\prime} u  \sum_{k=2}^{l-1} c_k \overline{d(k)}] (\text{mod }t) .\\
	\end{align*}
	Now,
	\begin{align*}
	\overline{ m_l} \gamma_l &= c_0 + c_1 \gamma_1 + \sum_{k=2}^{l-1} c_k \gamma_k \\
	&= M c_0^{\prime} + \overline{N} c_1^{\prime} \frac{Mu}{\overline{N} r^{\prime}} + \sum_{k=2}^{l-1} c_k [Mu \overline{d(k)} + \frac{M \overline{N} t \lambda_k}{d(k+1)}] \text{ where } \lambda_k \in \mathbb{Z}, \text{ by induction hypothesis}\\
	&= M \mathlarger{[}  \frac{c_0^{\prime}r^{\prime} + c_1^{\prime}u + r^{\prime} u \sum_{k=2}^{l-1} c_k \overline{d(k)}}{r^{\prime}}  + \frac{ \overline{N} t \theta_l}{d(l)}   \mathlarger{]} \text{ for some } \theta_l \in \mathbb{Z}, \text{ as } i_1 \leq i_2 \Longrightarrow d(i_1) \mid d(i_2) \\ 
	&= M \mathlarger{[} \frac{r^{\prime} u \overline{d(l+1)} + \mu_l t}{r^{\prime}} + \frac{ \overline{N}t \theta_l}{d(l)} \mathlarger{]} \text{ for some } \mu_l \in \mathbb{Z}\\
	&= Mu \overline{d(l+1)} + \frac{M \overline{N}t \mu_l}{\overline{ m_1}} + \frac{M \overline{N}t \theta_l}{d(l)} = Mu \overline{d(l+1)} + \frac{M \overline{N}t \lambda_l}{ d(l)} \text{ for some } \lambda_l \in \mathbb{Z}  \\
	&\Longrightarrow \gamma_l = Mu \overline{ m_2} \cdots \overline{ m_{l-1}} + \frac{M\overline{N}t \lambda_l}{ \overline{ m_1} \cdots \overline{ m_l} }.\\
	\end{align*}
	By our induction hypothesis, $(t, \overline{ m_k}) = 1 \, \forall \, k = 2, \cdots , l-1$. So $(p, \overline{ m_k}) = 1 $ for any prime $ p$ dividing $t, \, \forall \, k = 2, \cdots , l-1$, hence, $(p, \overline{d(l)}) = 1$. Suppose if possible $\exists$ a prime $p \mid t$ such that $p \mid \overline{ m_l}$. Let $\overline{ m_l} = p M_l$. Now, $(r^{\prime},t) = 1 \Longrightarrow (r^{\prime},p) = 1$. So $(p, r^{\prime} \overline{d(l)}) = 1$. So $\exists \, r_l \in \mathbb{Z}$ such that $r_l r^{\prime} \overline{d(l)} \equiv 1 (\text{mod }p)$. Let $r_l r^{\prime} \overline{d(l)} = 1 + \mu_l p$ for some $\mu_l \in \ZZ$. Now, 
	\begin{align*}
	\gamma_l &= Mu \overline{ m_2} \cdots \overline{ m_{l-1}} + \frac{M\overline{N}t \lambda_l}{ \overline{ m_1} \cdots \overline{ m_l} } \\
	&\Longrightarrow p M_l \gamma_l =  Mu \overline{ m_2} \cdots \overline{ m_{l}} + \frac{Mt \lambda_l}{r^{\prime} \overline{d(l)}} \text{ as } \overline{ m_{l}} = p M_l, \,  \overline{ m_1} = \overline{N}r^{\prime} , \, \overline{d(l)} = \prod_{k=2}^{l-1} \overline{ m_k}  \\
	&\Longrightarrow r^{\prime} \overline{d(l)} M_l \gamma_l =  r^{\prime} \overline{d(l)}Mu  \overline{ m_2} \cdots \overline{ m_{l-1}} M_l + M \lambda_l \frac{t}{p} \text{ as } \overline{ m_l} = pM_l\\
	&\Longrightarrow r_l r^{\prime} \overline{d(l)} M_l \gamma_l = r_l r^{\prime} \overline{d(l)}Mu\overline{ m_2} \cdots \overline{ m_{l-1}}M_l + r_l M \lambda_l \frac{t}{p} \in \ZZ\\
	&\Longrightarrow (1 + \mu_l p) M_l \gamma_l \in \ZZ  \Longrightarrow M_l \gamma_l + \mu_l \overline{ m_l} \gamma_l \in \mathbb{Z} \subset G(\gamma_0 , \cdots , \gamma_{l-1}) \Longrightarrow M_l \gamma_l \in G(\gamma_0 , \cdots , \gamma_{l-1}).
	\end{align*}
	But this contradicts the minimality of $\overline{ m_l}$. So for any prime $p$ dividing $t$, we have $p \nmid \overline{ m_l}$. Thus $(t, \overline{ m_l}) = 1$. We now have the induction step for $k = l$.
	\par  $(t, r^{\prime}) = 1 \Longrightarrow \overline{N}t \nmid \overline{N} r^{\prime} \Longrightarrow Nt \nmid \overline{N}r^{\prime}  \Longrightarrow \frac{n}{j} \nmid \overline{ m_1} \Longrightarrow \frac{n}{j} \nmid d(2)$. From the induction we have $(t, \overline{ m_k}) = 1 \, \forall \, k \geq 2$. Thus $ (t, \prod_{k=2}^{l-1} \overline{ m_k}) = 1 \Longrightarrow  (t, \overline{d(l)}) = 1 \, \forall \, l \geq 3 \Longrightarrow (t, r^{\prime} \overline{d(l)}) = 1 \, \forall \, l \geq 3$. $t \nmid r^{\prime} \overline{d(l)} \, \forall \, l \geq 3 \Longrightarrow \overline{N}t \nmid \overline{N} r^{\prime} \overline{d(l)} \, \forall \, l \geq 3 \Longrightarrow Nt \nmid \overline{ m_1} \overline{d(l)} \, \forall \, l \geq 3   \Longrightarrow \frac{n}{j} \nmid d(l) \, \forall \, l \geq 3$. So together we have, $\frac{n}{j} \nmid d(l) \, \forall \, l \geq 2$. Thus by Lemma \ref{not f.g condition}, we have $S^{R_{\mathfrak{m}}} (\nu)$ is not finitely generated over $S^{(A_{i,j,t,x})_{\mathfrak{n}}} (\nu)$.

\end{proof}

\par We are now ready to prove Theorem \ref{Main Theorem}.

\begin{proof}
	Let $i,j,t,x$ be positive integers satisfying the conditions of Theorem \ref{i,j,t,x theorem} and suppose that $\nu$ is a rational rank $1$ non discrete valuation dominating $R_{\mathfrak{m}}$ with a generating sequence (\ref{Gen Seq}) of eigenfunctions for $H_{i,j,t,x}$. By Proposition \ref{gcd > t}, we have $t \geq (\frac{m}{i}, \frac{n}{j})$. Since $t \mid \frac{m}{i}$ and $t \mid \frac{n}{j}$, we have $(\frac{m}{i}, \frac{n}{j}) = t$. 
	\par Conversely, let $i,j,t,x$ be positive integers satisfying the conditions of Theorem \ref{i,j,t,x theorem} and suppose that $(\frac{m}{i}, \frac{n}{j}) = t$. We will show that $\exists$ a rational rank $1$ non discrete valuation dominating $R_{\mathfrak{m}}$ with a generating sequence (\ref{Gen Seq}) of eigenfunctions for $H_{i,j,t,x}$. We consider the cases $t = 1$ and $t > 1$ separately. 
	\par Suppose that $(\frac{m}{i}, \frac{n}{j}) = t = 1$. We will construct a rational rank $1$ non discrete valuation $\nu$ dominating $R_{\mathfrak{m}}$, with a generating sequence (\ref{Gen Seq}) of eigenfunctions for $H_{i,j,t,x}$ such that $S^{R_{\mathfrak{m}}} (\nu)$ is finitely generated over $S^{(A_{i,j,t,x})_{\mathfrak{n}}} (\nu)$. Let $\{ q_l \}_{l \geq 2}$ be an infinite family of distinct prime numbers, such that $(q_l, \frac{m}{i}) = 1$, $(q_l, \frac{n}{j}) = 1$ for all $l \geq 2$. Let $q_1 = \frac{n}{j}$. Let $\{ c_l \}_{l \geq 1} \in \mathbb{Z}_{> 0}$ be positive integers such that 
	\begin{align*}
	c_1 &= \frac{m}{i} ,\, c_l \equiv 0 (\text{mod } \frac{m}{i}) \, \, \forall \, l \geq 1\\
	c_{l+1} &> q_{l+1} c_l \,\, \forall \, l  \geq 1 , \, (c_l, q_l) = 1  \, \, \forall \, l \geq 1 .\\
	\end{align*}
	We define a sequence of positive rational numbers $\{\gamma_l \}_{l \geq 0}$ as $\gamma_0 = 1$, $\gamma_l = \frac{c_l}{q_l} \, \, \forall \, l \geq 1$. We will show $\overline{m_l} = q_l \, \, \forall \, l \geq 1$, where $\overline{m_l} = $ min $\{q \in \mathbb{Z}_{> 0} \, | \, q \gamma_l \in G(\gamma_0 , \cdots , \gamma_{l-1}) \}$. Now, $\gamma_1 = \frac{c_1}{q_1}  = \frac{(\frac{m}{i})}{(\frac{n}{j})}$. Since $(\frac{m}{i}, \frac{n}{j}) = 1$, we have $\overline{m_1} = \frac{n}{j} = q_1$. For $l \geq 2$, $q_l \gamma_l = c_l \in \mathbb{Z} \Longrightarrow 1 \leq \overline{m_l} \leq q_l$. Suppose $q \in \mathbb{Z}_{> 0}$ such that $q \gamma_l  = q \frac{c_l}{q_l} = \sum_{k = 0}^{l-1} a_k \gamma_k = \sum_{k = 0}^{l-1} a_k \frac{c_k}{q_k}$. Then $q_l \mid q c_l \prod_{k =1}^{l-1} q_k$, that is, $q_l \mid q c_l \frac{n}{j} \prod_{k =2}^{l-1} q_k$. Now, $(q_l, c_l) = 1$ and $(q_l, \frac{n}{j}) = 1$. Again, $(q_l, q_k) = 1 \, \, \forall \, k \neq l$, as they are distinct primes. So, $q_l \, | \, q$. Thus we have $\overline{m_l} = q_l \, \, \forall \, l \geq 1$. And, $\overline{m_l} \gamma_l = q_l \gamma_l = c_l < \frac{c_{l+1}}{q_{l+1}} = \gamma_{l+1}$. Thus we have a sequence of positive rational numbers $\{ \gamma_l \}_{l \geq 0}$, such that $\gamma_{l+1} > \overline{m_l} \gamma_l  \,  \forall \, l \geq 1$. By Theorem $1.2$ of [\ref{CV1}], since $R_\mathfrak{m}$ is a regular local ring of dimension $2$, there is a valuation $\nu$ dominating $R_{\mathfrak{m}}$, such that $S^{R_{\mathfrak{m}}} (\nu) = S(\gamma_0, \gamma_1, \cdots)$. $\nu$ is a rational rank $1$ non discrete valuation by the construction. By Theorem $4.2$ of [\ref{CV1}], $\exists$ a generating sequence (\ref{Gen Seq}) $\{ Q_l \}_{l \geq 0}, Q_0 = X, Q_1 = Y, \cdots $ such that $\nu (Q_l) = \gamma_l \, \, \forall \, l \geq 0$. 
	\newline From the recursive construction of the $\{\gamma_l\}_{l \geq 0}$, we have the generating sequence as $Q_0 = X, \, Q_1 = Y , \, Q_2 = Y^{\frac{n}{j}} - \lambda_1 X^{\frac{m}{i}}$, where $\lambda_1 \in K \setminus \{ 0 \}$. For all $l \geq 2$, $Q_{l+1} = Q_{l}^{q_l} - \lambda_l X^{f_0} Y^{f_1} \cdots Q_{l-1}^{f_{l-1}}$, where $q_l \gamma_l = c_l = f_0 + \sum_{k = 1}^{l-1} f_k \gamma_k, \, 0 \leq f_k < \overline{ m_{k}} \, \forall \, k \geq 1$. Now, $(c_k, q_k) = 1  \, \forall \, k \geq 1$, and $(q_k, q_h) = 1 \, \forall \, k \neq h$. So, $c_l = f_0 + \sum_{k = 1}^{l - 1} \frac{f_k c_k}{q_k} \Longrightarrow c_l \prod_{k = 1}^{l-1} q_k = f_0 \prod_{k = 1}^{l-1} q_k + \frac{f_1 c_1 \prod_{k = 1}^{l-1} q_k}{q_1} + \cdots + \frac{f_{l-1} c_{l-1} \prod_{k = 1}^{l-1} q_k}{q_{l-1}}$, which implies $q_k \mid f_k \, \forall \, k \geq 1$. Since $0 \leq f_k < \overline{ m_{k}} = q_k$, this implies $f_k = 0 \, \forall \, k \geq 1$. So we have the generating sequence as,	
	\begin{align*}
	Q_0 = X, \, Q_1 = Y , \, Q_2 &= Y^{\frac{n}{j}} - \lambda_1 X^{\frac{m}{i}} , \, Q_{l+1} = Q_{l}^{q_l} - \lambda_{l} X^{c_l} \, \, \forall \, l \geq 2 \\
	&\text{where } \lambda_l \in K \setminus \{ 0 \} \, \forall \, l \geq 1.
	\end{align*}
	We now show that each $Q_l$ is an eigenfunction for $H_{i,j,1,1}$. $H_{i,j,1,1} = \{ (\alpha^{ai}, \beta^{bj}) \mid a,b \in \mathbb{Z} \}$. For all $l \geq 2$, $ d(l) = \prod_{k = 1}^{l-1} \overline{m_k} =   q_1 \cdots q_{l-1}  = \frac{n}{j} q_2 \cdots q_{l-1}$. We have, $(\alpha^{ai}, \beta^{bj}) \cdot Q_2 = \beta^{bj \frac{n}{j}}  Y^{\frac{n}{j}} -  \lambda_1 \alpha^{ai\frac{m}{i}} X^{\frac{m}{i}} = Q_2$. So, $Q_2$ is an eigenfunction. Suppose $Q_3, \cdots , Q_l$ are eigenfunctions for $H_{i,j,1,1}$. We check for $Q_{l+1}$.
	From (\ref{Q_m as eigenfunction}), $(\alpha^{ai}, \beta^{bj}) \cdot Q_k = \beta^{bj d(k)} Q_k \, \, \forall \, 2 \leq k \leq l$. Since $\frac{m}{i} \mid c_l$ and $\frac{n}{j} \mid d(l)$, we have $(\alpha^{ai}, \beta^{bj}) \cdot Q_{l+1} = \beta^{bjq_l d(l)} Q_l^{q_l} - \lambda_l \alpha^{ai c_l} X^{c_l} =Q_{l+1}$. Thus $Q_{l+1}$ is an eigenfunction. Thus by induction, $\{Q_l\}_{l \geq 0}$ is a generating sequence of eigenfunctions for $H_{i,j,1,1}$.

	\vspace{2 mm}
	\par Now we consider the case $(\frac{m}{i}, \frac{n}{j}) = t > 1$. We will construct a rational rank $1$ non discrete valuation $\nu$ dominating $R_{\mathfrak{m}}$, with a generating sequence (\ref{Gen Seq}) of eigenfunctions for $H_{i,j,t,x}$. 
	\par Since $(t,x) = 1$, there are positive integers $r,s$ such that $rx - st = 1$. So $(r,t) = 1$. From Lemma $3$ in $\S 2$, Chapter $III$ of [\ref{Serre}], we have that if $r,t$ are positive integers such that $(r,t) = 1$, then there are infinitely many prime numbers of the form $r + \theta t$, where $\theta \in \NN$. Define the family $ \mathfrak{R} =  \{ r^{(k)} \}_{k \geq 0}$ as $r^{(0)} = r, \, r^{(k)}$ = $k$-th prime in the above prime series. Any two elements in the family $\mathfrak{R}$ are coprime by construction. Also, $r^{(k)} = r + \theta_k t \Longrightarrow r^{(k)} \equiv r (\text{mod }t) \, \, \forall \, k$. Since $\mathfrak{R}$ is an infinite family such that any two elements in $\mathfrak{R}$ are mutually prime, it follows that there is an infinite ordered family of distinct prime numbers  $ \mathfrak{F} =  \{r_l \}_{l \geq 1}$ such that, $r_l \equiv r (\text{mod }t)$, $(r_l,{\frac{(\frac{m}{i})}{t}}) = 1$, $(r_l, \frac{(\frac{n}{j})}{t}) = 1$, $(r_l, w_1) = 1$, $(r_l, w_2) = 1 \, \forall \, l \geq 1$, where $w_1$ and $w_2$ are as in Remark \ref{relation between alpha and beta}. Let $d = (w_1, w_2)$. Thus $(\frac{w_1}{d}, \frac{w_2}{d}) = 1$. Define two sequences $(a_l)_{l \geq 1}$ and $(b_l)_{l \geq 1}$ of non negative integers as follows,
	\begin{align*}
	b_1 &= 0 , \, r_l \mid b_l  \, \, \forall \, l \geq 2  , \, t \mid b_l  \, \, \forall \, l \geq 2 \\
	b_{l+1}  &> r_{l+1} [r^{l-1} + b_l] - r^l  \, \, \forall \, l \geq 1\\
	a_l &= \frac{(\frac{m}{i})}{t} [r^{l-1} + b_l]\frac{w_2}{d}  \, \, \forall \, l \geq 1 .\\
	\end{align*}
	Here $r_l \in \mathfrak{F} \, \forall \, l \geq 1$. Define a sequence of positive rational numbers $\{\gamma_l \}_{l \geq 0}$ as follows 
	\begin{align*}
	\gamma_0 = 1, \, \gamma_1 &= \frac{ \frac{(\frac{m}{i})}{t} \frac{w_2}{d} }{r_1 \frac{(\frac{n}{j})}{t} \frac{w_1}{d} }, \\
	\gamma_l &= \frac{a_l}{r_l} = \frac{(\frac{m}{i})}{t} [\frac{r^{l-1} + b_l}{r_l}]\frac{w_2}{d} \,   \, \,\forall \, l \geq 2 .\\ 
	\end{align*}  
	We will show $\overline{m_1} = r_1  \frac{(\frac{n}{j})}{t} \frac{w_1}{d}$ and $\overline{m_l} = r_l \, \, \forall \, l \geq 2$, where $\overline{m_l}$ = min $\{q \in \mathbb{Z}_{> 0} \, | \, q \gamma_l \in G(\gamma_0 , \cdots , \gamma_{l-1}) \}$. $(\frac{w_1}{d}, \frac{w_2}{d}) = 1$, $(r_1, \frac{w_2}{d}) = 1$ and $(\frac{(\frac{n}{j})}{t}, \frac{w_2}{d}) = 1$ implies $(\frac{w_2}{d}, r_1  \frac{(\frac{n}{j})}{t} \frac{w_1}{d}) = 1$. Also, $(\frac{(\frac{m}{i})}{t}, \frac{(\frac{n}{j})}{t}) = 1$, $(\frac{(\frac{m}{i})}{t}, r_1) = 1$ and $(\frac{(\frac{m}{i})}{t}, \frac{w_1}{d}) = 1$ implies $(\frac{(\frac{m}{i})}{t}, r_1  \frac{(\frac{n}{j})}{t} \frac{w_1}{d}) = 1$. Thus, $(\frac{w_2}{d} \frac{(\frac{m}{i})}{t}, r_1  \frac{(\frac{n}{j})}{t} \frac{w_1}{d}) = 1$, hence $\overline{m_1} = r_1  \frac{(\frac{n}{j})}{t} \frac{w_1}{d}$.
	\newline Now $\forall \, l \geq 2$, $r_l \gamma_l = a_l \in \mathbb{Z} \Longrightarrow 1 \leq \overline{m_l} \leq r_l$. Suppose $ \exists$ a positive integer $q $ such that $q \gamma_l \in G(\gamma_0, \cdots, \gamma_{l-1})$. Then $q \gamma_l = q \frac{a_l}{r_l} = c_0 + c_1 \frac{a_1}{r_1 \frac{(\frac{n}{j})}{t} \frac{w_1}{d}} + \sum_{k = 2 }^{l-1} c_k \frac{a_k}{r_k}$, where $c_k \in \mathbb{Z} \, \forall \, k = 0, \cdots , l-1$. Thus $r_l \mid q a_l  \frac{(\frac{n}{j})}{t} \frac{w_1}{d} \prod_{k=1}^{l-1} r_k$. Now, $(r_l, \frac{(\frac{n}{j})}{t}) = 1$, and $(r_l, r_k) = 1 \, \forall \, k \neq l$, as they are distinct primes. Also, $(r_l, \frac{w_1}{d}) = 1$. So, $ r_l \mid q a_l$. And, $r_l > r \Longrightarrow r_l \nmid r \Longrightarrow r_l \nmid \frac{(\frac{m}{i})}{t} [r^{l-1} + b_l] \frac{w_2}{d} = a_l$ as $(r_l, \frac{w_2}{d}) = 1$, $(r_l,\frac{(\frac{m}{i})}{t} ) = 1$ and $r_l \mid b_l$. Thus, $r_l \mid q$. Hence we have $\overline{m_1} = r_1  \frac{(\frac{n}{j})}{t} \frac{w_1}{d}$ and $\overline{m_l} = r_l \, \, \forall \, l \geq 2$.
	\newline  Now, $b_{l + 1} > r_{l+1} [r^{l-1} + b_l] - r^l \, \forall \, l \geq 1$ and $b_1 = 0$ implies $b_2 > r_2 - r$. Thus, $a_2 = \frac{(\frac{m}{i})}{t} [r + b_2] \frac{w_2}{d} > r_2 \frac{(\frac{m}{i})}{t} \frac{w_2}{d} \Longrightarrow \gamma_2 = \frac{a_2}{r_2} > \frac{(\frac{m}{i})}{t} \frac{w_2}{d} = \overline{ m_1} \gamma_1$. For $l \geq 2$, we have $r^l + b_{l+1} > r_{l+1} [r^{l-1} + b_l] \Longrightarrow \frac{(\frac{m}{i})}{t} [r^l + b_{l+1}] \frac{w_2}{d} >  r_{l+1} \frac{(\frac{m}{i})}{t} [r^{l-1} + b_l] \frac{w_2}{d} \Longrightarrow \gamma_{l+1} = \frac{a_{l+1}}{r_{l+1}} > a_l = \overline{ m_l} \gamma_l$.
	\newline Thus we have a sequence of positive rational numbers $\{ \gamma_l \}_{l \geq 0}$ such that $\gamma_{l+1} > \overline{m_l} \gamma_l  \,  \forall \, l \geq 1$. By Theorem $1.2$ of [\ref{CV1}], since $R_\mathfrak{m}$ is a regular local ring of dimension $2$, there is a valuation $\nu$ dominating $R_{\mathfrak{m}}$, such that $S^{R_{\mathfrak{m}}} (\nu) = S(\gamma_0, \gamma_1, \cdots)$. $\nu$ is a rational rank $1$ non discrete valuation by the construction. By Theorem $4.2$ of [\ref{CV1}], $\exists$ a generating sequence (\ref{Gen Seq}) $\{ Q_l \}_{l \geq 0}, Q_0 = X, Q_1 = Y, \cdots $ such that $\nu (Q_l) = \gamma_l \, \, \forall \, l \geq 0$. 
	\newline From the recursive construction of the $\{ \gamma_l \}_{l \geq 0}$, we have the generating sequence as $Q_0 = X, \, Q_1 = Y, \, Q_2 = Y^{r_1 \frac{(\frac{n}{j})}{t} \frac{w_1}{d}} - \lambda_1 X^{\frac{(\frac{m}{i})}{t} \frac{w_2}{d}}$. For all $l \geq 2$, $Q_{l+1} = Q_l^{r_l} - \lambda_l X^{f_0} Y^{f_1} \cdots Q_{l-1}^{f_{l-1}}$, where $0 \leq f_k < \overline{ m_k} \, \forall \, k \geq 1$ and $r_l \gamma_l = a_l = f_0 + \sum_{k = 1}^{l-1} f_k \gamma_k$. So, $a_l = f_0 + \sum_{k = 1}^{l-1} \frac{f_k a_k}{\overline{ m_k}}$. We observe, from our construction, $(\overline{ m_k}, \overline{ m_h}) = 1 \, \forall \, k \neq h$. Also, $(\overline{ m_{k}}, a_k) = 1 \, \forall \, k \geq 1$. 
	\newline Thus, $a_l \prod_{ k = 1}^{l-1} \overline{ m_k} = f_0 \prod_{ k = 1}^{l-1} \overline{ m_k} + \frac{f_1 a_1 \prod_{ k = 1}^{l-1} \overline{ m_k}}{\overline{ m_1}} + \cdots + \frac{f_{l-1} a_{l-1} \prod_{ k = 1}^{l-1} \overline{ m_k}}{\overline{ m_{l-1}}} \Longrightarrow \overline{ m_k} \mid f_k \, \forall \, k \geq 1$. Since $0 \leq f_k < \overline{ m_k}$, we have $f_k = 0 \, \forall \, k \geq 1$. Thus the generating sequence is given as,
	\begin{align*}
	Q_0 = &X, \, Q_1 = Y, \, Q_2 = Y^{r_1 \frac{(\frac{n}{j})}{t} \frac{w_1}{d}} - \lambda_1 X^{\frac{(\frac{m}{i})}{t} \frac{w_2}{d}} \\
	&Q_{l+1} = Q_l^{r_l} - \lambda_l X^{a_l} \, \, \forall \, l \geq 2 \\
	&\text{where } \lambda_l \in K \setminus \{ 0 \} \, \forall \, l \geq 1. \\ 
	\end{align*}	
	This is a minimal generating sequence as $\overline{ m_l} > 1 \, \forall \, l \geq 1 $. We now show that each $Q_l$ is an eigenfunction for $H_{i,j,t,x}$. 	
	From (\ref{K-algebra action}), $(\alpha^{ai}, \beta^{bj}) \cdot Q_2 = \beta^{\frac{r_1 b n}{t} \frac{w_1}{d}} Y^{r_1 \frac{(\frac{n}{j})}{t} \frac{w_1}{d}} - \lambda_1 \alpha^{\frac{am}{t} \frac{w_2}{d}}  X^{\frac{(\frac{m}{i})}{t} \frac{w_2}{d}}$. Now, $\forall \, b \equiv ax (\text{mod }t)$, $r_1 b \equiv a (\text{mod }t)$, hence, $(\frac{r_1 b - a}{t}) (\frac{w_1 w_2}{d}) \in \mathbb{Z}$. Thus by Remark \ref{relation between alpha and beta}, $\beta^{\frac{r_1 bn}{t} \frac{w_1}{d}} = \alpha^{\frac{am}{t} \frac{w_2}{d}} \, \forall \, b \equiv ax (\text{mod } t)$, that is, $Q_2$ is an eigenfunction for $H_{i,j,t,x}$.
	\newline Suppose $Q_3, \cdots , Q_l$ are eigenfunctions for $H_{i,j,t,x}$. We check for $Q_{l+1}$. We note $d(k) = \overline{ m_1} \cdots \overline{ m_{k-1}} =  \frac{(\frac{n}{j})}{t} \frac{w_1}{d} r_1 r_2 \cdots r_{k-1}$. From (\ref{Q_m as eigenfunction}) we have, $(\alpha^{ai}, \beta^{bj}) \cdot Q_k = \beta^{bj d(k)} Q_k \, \, \forall \, 1 \leq k \leq l$. Now, $(\alpha^{ai}, \beta^{bj}) \cdot Q_{l+1} = \beta^{\frac{b n r_1 \cdots r_l}{t} \frac{w_1}{d}}Q_l^{r_l} - \lambda_l \alpha^{a i a_l} X^{a_l}$. Since $r_k \equiv r (\text{mod }t) \, \forall \, k \geq 1$, $rx \equiv 1 (\text{mod }t)$ and $t \mid b_l$, we have
	\begin{align*}
	&\frac{b r_1 \cdots r_l}{t} - \frac{a r^{l-1}}{t} \in \ZZ \, \forall \, b \equiv ax (\text{mod }t) \\
	\Longrightarrow  &\frac{b r_1 \cdots r_l}{t} - \frac{a [r^{l-1} + b_l]}{t} \in \ZZ \, \forall \, b \equiv ax (\text{mod }t) \\
	\Longrightarrow &\frac{b r_1 \cdots r_l}{t} (\frac{w_1 w_2}{d}) - \frac{a [r^{l-1} + b_l]}{t} (\frac{w_1 w_2}{d}) \in \ZZ \, \forall \, b \equiv ax (\text{mod }t) \\
	\Longrightarrow & \frac{bn r_1 \cdots r_l}{t} (\frac{w_1 w_2}{dn}) - \frac{ai (\frac{m}{i}) [r^{l-1} + b_l]}{t} (\frac{w_1 w_2}{dm}) \in \ZZ \, \forall \, b \equiv ax (\text{mod }t) \\
	\Longrightarrow & (\frac{b n r_1 \cdots r_l}{t} \frac{w_1}{d}) \frac{w_2}{n} - (a i a_l) \frac{w_1}{m} \in \ZZ \, \forall \, b \equiv ax (\text{mod }t) .\\
	\end{align*}
	Thus, by Remark \ref{relation between alpha and beta}, $\beta^{\frac{b n r_1 \cdots r_l}{t} \frac{w_1}{d}} = \alpha^{a i a_l}$ for all $b \equiv ax (\text{mod }t)$, and hence $Q_{l+1}$ is an eigenfunction for $H_{i,j,t,x}$. Thus by induction, $\{Q_l\}_{l \geq 0}$ is a minimal generating sequence of eigenfunctions for $H_{i,j,t,x}$. This completes the proof of part $1)$ of Theorem \ref{Main Theorem}. 
	
	\vspace{2 mm}
	\par Now we suppose $(\frac{m}{i}, \frac{n}{j}) = t = 1$ and $\nu$ is a rational rank $1$ non discrete valuation dominating $R_{\mathfrak{m}}$ with a generating sequence (\ref{Gen Seq}) of eigenfunctions for $H_{i,j,1,1}$. Let $\nu (Q_l) = \gamma_l \, \forall \, l \in \NN$. We have $Q_0 = X, Q_1 = Y$. By Equation ($8$) in [\ref{CV1}], $Q_2 = Y^s - \lambda X^r$ where $\lambda \in K \setminus \{ 0 \}$, $s \gamma_1 = r \gamma_0$. Since $(\frac{m}{i}, \frac{n}{j})= 1$, by Chinese Remainder Theorem (Theorem $2.1$, $\S 2$, [\ref{Lang}]) we have $H_{i,j,1,1}$ is a cyclic group, generated by $(\alpha^i, \beta^j)$. By (\ref{K-algebra action}) we have $(\alpha^i, \beta^j) \cdot Q_2 = \beta^{sj}Y^s - \lambda \alpha^{ir} X^r$. Since $Q_2$ is an eigenfunction, we have 
	\begin{align*}
	\beta^{sj} = \alpha^{ir} &\Longrightarrow \frac{sjw_2}{n} - \frac{irw_1}{m} \in \ZZ \text{ by Remark } \ref{relation between alpha and beta}\\
	&\Longrightarrow \frac{sw_2}{(\frac{n}{j})} - \frac{rw_1}{(\frac{m}{i})} \in \ZZ \\
	&\Longrightarrow \frac{m}{i} \mid r \text{ and } \frac{n}{j} \mid s \text{ as } (\frac{m}{i}, w_1) = 1, \, (\frac{n}{j}, w_2) = 1, \, (\frac{m}{i}, \frac{n}{j}) = 1.
	\end{align*}
	So, $Q_2 = Y^s - \lambda X^r \in K[X^{\frac{m}{i}}, Y^{\frac{n}{j}}] \subset A_{i,j,1,1}$. Thus by Proposition \ref{f.g iff all q_i after some i}, we have part $2)$ of Theorem \ref{Main Theorem}.
	
	\vspace{2 mm}
	\par We observe that the part $3)$ of Theorem \ref{Main Theorem} follows from Propositions \ref{p nmid N} and \ref{p divides N}. This completes the proof of Theorem \ref{Main Theorem}.

	\section{ Non-splitting} Suppose that a local domain $B$ dominates a local domain $A$. Let $L$ be the quotient field of $A$ and $M$ be the quotient field of $B$. Suppose $\omega$ is a valuation of $L$ which dominates $A$. We say that $\omega$ does not split in $B$ if there is a unique extension $\omega^*$ of $\omega$ to $M$ which dominates $B$. 
	\par  We use the same notation as in the previous sections.
	
	\begin{Theorem}\label{non splitting}
		Let $i,j,t,x$ be positive integers satisfying the conditions of Theorem \ref{i,j,t,x theorem} such that $(\frac{m}{i}, \frac{n}{j}) = t$. Suppose that $\nu$ is a rational rank $1$ non discrete valuation dominating $R_{\mathfrak{m}}$ with a generating sequence (\ref{Gen Seq}) of eigenfunctions for $H_{i,j,t,x}$. Let $\overline{ \nu} = \nu \mid_{Q (A_{i,j,t,x})}$ where $Q(A_{i,j,t,x})$ denotes the quotient field of $A_{i,j,t,x}$. Then $\overline{ \nu}$ does not split in $R_{\mathfrak{m}}$.
	\end{Theorem} 
	
	\begin{proof}
		Let $\{ Q_k \}_{k \geq 0}$, $\{ \gamma_k \}_{k \geq 0}$ and $\{ \overline{ m_{k}} \}_{k \geq 1}$ be as in Section \ref{GS}. Thus $Q_0 = X$ and $Q_1 = Y$. Without any loss of generality, we can assume $\gamma_0 = 1$. Set $\frac{m}{i} = Mt$ and $\frac{n}{j} = Nt$ where $M,N \in \mathbb{Z}_{> 0}$ and $(M,N) = 1$. From (\ref{invariant subsemigroup formula}) we have 
		\begin{equation*}
		S^{(A_{i,j,t,x})_{\mathfrak{n}}} (\nu)= \left\{ l \gamma_0 + j_1 \gamma_1 + \cdots + j_r \gamma_r  \, \left | \, \begin{array}{@{}l@{\thinspace}l} 
		l \in \mathbb{N}, \, r \in \NN,   \, 0 \leq j_k < \overline{m_k} \, \forall \, k = 1, \cdots , r\\
		\alpha^{lai}\beta^{bj\mathlarger{\sum_{k=1}^r [j_{k} d(k)] }} = 1\\
		\forall \, b \equiv ax (\text{mod }t)\\
		\end{array}
		\right.
		\right\}.
		\end{equation*}	
		Now, $\overline{ \nu} = \nu \mid_{Q(A_{i,j,t,x})}$. Thus $S^{(A_{i,j,t,x})_{\mathfrak{n}}} (\nu)= \{ \nu (f) \mid 0 \neq f \in (A_{i,j,t,x})_{\mathfrak{n}}  \} = S^{(A_{i,j,t,x})_{\mathfrak{n}}} (\overline{ \nu})$. The group generated by $S^{(A_{i,j,t,x})_{\mathfrak{n}}} (\overline{ \nu})$ is $\Gamma_{\overline{ \nu}}$, the value group of $\overline{ \nu}$ ($1.2$, [\ref{Cut_Ramification of Valuations}]). Thus $\Gamma_{\overline{ \nu}} = \{ s_1 - s_2  \mid s_1, s_2 \in S^{(A_{i,j,t,x})_{\mathfrak{n}}} (\nu) \}$. Suppose $\gamma_0 \in \Gamma_{\overline{ \nu}}$. Then we have a representation, 
		\begin{align*}
		\gamma_0 = (l_1 \gamma_0 + \sum_{k = 1}^{r} h_{1,k} \gamma_k) - (l_2 \gamma_0 + \sum_{k = 1}^{r} h_{2,k} \gamma_k) = (l_1 - l_2) \gamma_0 + \sum_{k = 1}^{r} (h_{1,k} - h_{2,k}) \gamma_k
		\end{align*}
		where $l_1 \gamma_0 + \sum_{k = 1}^{r} h_{1,k} \gamma_k \in S^{(A_{i,j,t,x})_{\mathfrak{n}}} (\nu)$, and $l_2 \gamma_0 + \sum_{k = 1}^{r} h_{2,k} \gamma_k \in S^{(A_{i,j,t,x})_{\mathfrak{n}}} (\nu)$. Thus $l_1, l_2 \in \NN$, $r \in \NN$ and $0 \leq h_{1,k}, h_{2,k} < \overline{m_k} \, \forall \, k = 1 , \cdots , r$. So, $|h_{1,k} - h_{2,k}| < \overline{m_k} \, \forall \, k = 1 , \cdots , r$. Now $(h_{1,r} - h_{2,r}) \gamma_r \in G(\gamma_0, \cdots , \gamma_{r-1})$ and $|h_{1,r} - h_{2,r}| < \overline{m_r} \Longrightarrow h_{1,r} = h_{2,r}$. With the same argument, we have $h_{1,k} = h_{2,k} \, \forall \, k = 1 , \cdots , r$. So in the representation of $\gamma_0$, we have $\gamma_0 = (l_1 - l_2) \gamma_0 \Longrightarrow l_1 - l_2 = 1$. Also, 
		\begin{align*}
		&\alpha^{l_1 ai}\beta^{bj\mathlarger{\sum_{k=1}^r [h_{1,k} d(k)] }} = 1 = \alpha^{l_2 ai}\beta^{bj\mathlarger{\sum_{k=1}^r [h_{2,k} d(k)] }} \\	
		&\Longrightarrow \alpha^{(l_1 - l_2) a i} \beta^{bj \sum_{k=1}^{r} [(h_{1,k} - h_{2,k}) d(k) ]} = 1 \, \forall \, b \equiv ax (\text{mod }t). \\
		\end{align*}
		Since $l_1 - l_2 = 1$ and $h_{1,k} = h_{2,k} \, \forall \, k = 1 , \cdots , r$, we have $\alpha^{ai} = 1 \, \forall \, b \equiv ax (\text{mod }t)$. Thus $\alpha^i = 1$, hence, $m \mid i$, that is, $m = i$. So we have obtained, 
		\begin{equation}\label{beta0 in value group}
		\gamma_0 \in \Gamma_{\overline{ \nu}} \Longrightarrow M = 1, \, t = 1.
		\end{equation}
		Suppose $\gamma_1 \in \Gamma_{\overline{ \nu}}$. Then we have a representation,
		\begin{align*}
		\gamma_1 = (l_1 \gamma_0 + \sum_{k = 1}^{r} j_{1,k} \gamma_k) - (l_2 \gamma_0 + \sum_{k = 1}^{r} j_{2,k} \gamma_k) = (l_1 - l_2) \gamma_0 + \sum_{k = 1}^{r} (j_{1,k} - j_{2,k}) \gamma_k
		\end{align*} 
		where $l_1 \gamma_0 + \sum_{k = 1}^{r} j_{1,k} \gamma_k \in S^{(A_{i,j,t,x})_{\mathfrak{n}}} (\nu)$, and $l_2 \gamma_0 + \sum_{k = 1}^{r} j_{2,k} \gamma_k \in S^{(A_{i,j,t,x})_{\mathfrak{n}}} (\nu)$. So, $l_1, l_2 \in \NN$, $r \in \NN$ and $0 \leq j_{1,k}, j_{2,k} < \overline{m_k} \, \forall \, k = 1, \cdots , r$. So, $|j_{1,k} - j_{2,k}| < \overline{m_k} \, \forall \, k = 1 , \cdots , r $. Now, $(j_{1,r} - j_{2,r}) \gamma_r \in G(\gamma_0, \cdots , \gamma_{r-1}) $ and $|j_{1,r} - j_{2,r}| < \overline{m_r} \Longrightarrow j_{1,r} = j_{2,r}$. With the same argument, we have $j_{1,k} = j_{2,k} \, \forall \, k = 2, \cdots r$. Thus we have, $\gamma_1 = (l_1 - l_2) \gamma_0 + (j_{1,1} - j_{2,1}) \gamma_1$ where $0 \leq |j_{1,1} - j_{2,1}| < \overline{m_1}$. Again, $\forall \, b \equiv ax (\text{mod }t)$ we have
		\begin{align*}
		\alpha^{l_1 ai}\beta^{bj\mathlarger{\sum_{k=1}^r [j_{1,k} d(k)] }} = 1 = \alpha^{l_2 ai}\beta^{bj\mathlarger{\sum_{k=1}^r [j_{2,k} d(k)] }}.
		\end{align*}
		Since $d(1) = $ deg$_Y (Y) = 1$ and $j_{1,k} = j_{2,k} \, \forall \, k = 2, \cdots , r$, we have $\alpha^{(l_1 - l_2) a i} \beta^{bj (j_{1,1} - j_{2,1})} = 1$ for all $b \equiv ax (\text{mod }t)$. So if $\gamma_1 \in \Gamma_{\overline{ \nu}}$, we have a representation 
		\begin{equation*}
		\begin{split}
		\gamma_1 = l \gamma_0 + j_1 \gamma_1 \text{ where } &l \in \ZZ, \, 0 \leq |j_1| < \overline{m_1}\\
		&\alpha^{lai} \beta^{bj j_1} = 1 \, \forall \, b \equiv ax (\text{mod }t). 
		\end{split}
		\end{equation*}
		In the above expression, $(1-j_1) \gamma_1 = l \gamma_0 \in \gamma_0 \ZZ \Longrightarrow \overline{m_1} \mid (1 - j_1)$. 
		\newline And $|1-j_1| \leq 1 + |j_1| \leq \overline{m_1} \Longrightarrow |1-j_1| = 0$ or $\overline{m_1}$. $1-j_1 = 0 \Longrightarrow l = 0, \, j_1 = 1$. From the above expression we then have, $\beta^{bj} = 1 \, \forall \, b \equiv ax (\text{mod }t) \Longrightarrow  n = j$. Now consider $|1 - j_1| = \overline{m_1}$. If $1 - j_1 = - \overline{m_1}$ then $j_1 = 1 + \overline{m_1}$ which contradicts $|j_1| < \overline{m_1}$. So $1 - j_1 = \overline{m_1}$, that is, $j_1 = 1 - \overline{m_1}$. And $(1 - j_1) \gamma_1 = \overline{m_1} \gamma_1 = l \gamma_0$. So $Q_2 = Q_1^{\overline{m_1}} - \lambda X^l$ where $\lambda \in K \backslash \{ 0\}$. $(\alpha^{ai}, \beta^{bj}) \cdot Q_2 = \beta^{bj \overline{m_1} } Q_1^{\overline{m_1}}  - \lambda \alpha^{ail}  X^l$. Since $Q_2$ is an eigenfunction, we have $\beta^{bj \overline{m_1}} = \alpha^{ail} \, \forall \, b \equiv ax (\text{mod }t)$.
		Again from the above expression we have, $\alpha^{ail} \beta^{bj} = \beta^{bj \overline{m_1}} \, \forall \, b \equiv ax (\text{mod }t)$, as $j_1 = 1 - \overline{m_1}$. Thus, $\beta^{bj} = 1 \, \forall \, b \equiv ax (\text{mod }t)$, and hence $j = n$. So we have obtained, 
		\begin{equation}\label{beta1 in value group}
		\gamma_1 \in \Gamma_{\overline{ \nu}} \Longrightarrow N = 1, \, t = 1.
		\end{equation} 
		For an element $g \in \Gamma_\nu$, let $[g]$ denote the class of $g$ in $\frac{\Gamma_\nu}{\Gamma_{\overline{ \nu}}}$. Since $\frac{\Gamma_\nu}{\Gamma_{\overline{ \nu}}}$ is a finite group, $[g]$ has finite order for each $g \in \Gamma_\nu$. Let $e = [\Gamma_\nu : \Gamma_{\overline{ \nu}}]$.
		\par  
		First we suppose $\gamma_0 \in \Gamma_{\overline{ \nu}}$ and $\gamma_1 \in \Gamma_{\overline{ \nu}}$. From (\ref{beta0 in value group}) and (\ref{beta1 in value group}) we have $M = N = t = 1$. From Proposition \ref{order of group} we have $|H_{i,j,t,x}| = MNt = 1$. Thus, $MNt \mid e$.
		\par Now we suppose $\gamma_0 \notin \Gamma_{\overline{ \nu}}$ and $\gamma_1 \in \Gamma_{\overline{ \nu}}$. From (\ref{beta1 in value group}) we have $N = t = 1$. From Proposition \ref{order of group} we have $|H_{i,j,t,x}| = MNt = M$. Let $f_0$ denote the order of $[\gamma_0]$. Thus $f_0 \gamma_0 \in \Gamma_{\overline{ \nu}}$. We thus have a representation 
		\begin{align*}
		f_0 \gamma_0 = (l_1 \gamma_0 + \sum_{k = 1}^{r} h_{1,k} \gamma_k) - (l_2 \gamma_0 + \sum_{k = 1}^{r} h_{2,k} \gamma_k) = (l_1 - l_2) \gamma_0 + \sum_{k = 1}^{r} (h_{1,k} - h_{2,k}) \gamma_k
		\end{align*} 
		where $l_1 \gamma_0 + \sum_{k = 1}^{r} h_{1,k} \gamma_k \in S^{(A_{i,j,t,x})_{\mathfrak{n}}} (\nu)$, and $l_2 \gamma_0 + \sum_{k = 1}^{r} h_{2,k} \gamma_k \in S^{(A_{i,j,t,x})_{\mathfrak{n}}} (\nu)$. Thus $l_1, l_2 \in \NN$, $r \in \NN$ and $0 \leq h_{1,k}, h_{2,k} < \overline{m_k} \, \forall \, k = 1 , \cdots , r$. So, $|h_{1,k} - h_{2,k}| < \overline{m_k} \, \forall \, k = 1 , \cdots , r$.
		With the same arguments as above, we have $h_{1,k} = h_{2,k} \, \forall \, k = 1 , \cdots , r$. Thus $f_0 \gamma_0 = (l_1 - l_2) \gamma_0 \Longrightarrow f_0 = l_1 - l_2$. And, for all $ b \equiv ax (\text{mod }t)$, 
		\begin{equation*}
		\alpha^{l_1 ai}\beta^{bj\mathlarger{\sum_{k=1}^r [h_{1,k} d(k)] }} = 1 = \alpha^{l_2 ai}\beta^{bj\mathlarger{\sum_{k=1}^r [h_{2,k} d(k)] }}.
		\end{equation*}
		So, $\alpha^{(l_1 - l_2)i} = \alpha^{f_0 i} = 1$, hence $Mt \mid f_0 \Longrightarrow Mt \mid e$. Thus $MNt \mid e \text{ as } MNt = M$.	
		\par Now we suppose $\gamma_0 \in \Gamma_{\overline{ \nu}}$ and $\gamma_1 \notin \Gamma_{\overline{ \nu}}$. From (\ref{beta0 in value group}) we have $M = t = 1$. $|H_{i,j,t,x}| = MNt = N$. Let $f_1$ denote the order of $[\gamma_1]$, that is $f_1 \gamma_1 \in \Gamma_{\overline{ \nu}}$. We have a representation,
		\begin{align*}
		f_1 \gamma_1 = (l_1 \gamma_0 + \sum_{k = 1}^{r} j_{1,k} \gamma_k) - (l_2 \gamma_0 + \sum_{k = 1}^{r} j_{2,k} \gamma_k) = (l_1 - l_2) \gamma_0 + \sum_{k = 1}^{r} (j_{1,k} - j_{2,k}) \gamma_k
		\end{align*} 
		where $l_1 \gamma_0 + \sum_{k = 1}^{r} j_{1,k} \gamma_k \in S^{(A_{i,j,t,x})_{\mathfrak{n}}} (\nu)$, and $l_2 \gamma_0 + \sum_{k = 1}^{r} j_{2,k} \gamma_k \in S^{(A_{i,j,t,x})_{\mathfrak{n}}} (\nu)$. So, $l_1, l_2 \in \NN$, $r \in \NN$ and $0 \leq j_{1,k}, j_{2,k} < \overline{m_k} \, \forall \, k = 1, \cdots , r$. So, $|j_{1,k} - j_{2,k}| < \overline{m_k} \, \forall \, k = 1 , \cdots , r. $ With the same arguments as above, we have $j_{1,k} = j_{2,k} \, \forall \, k = 2, \cdots , r$. So in the above representation, we have $f_1 \gamma_1 = (l_1 - l_2) \gamma_0 + (j_{1,1} - j_{2,1}) \gamma_1$ where $0 \leq |j_{1,1} - j_{2,1}| < \overline{m_1}$. Again, $\forall \, b \equiv ax (\text{mod }t)$ we have
		\begin{equation*}
		\alpha^{l_1 ai}\beta^{bj\mathlarger{\sum_{k=1}^r [j_{1,k} d(k)] }} = 1 = \alpha^{l_2 ai}\beta^{bj\mathlarger{\sum_{k=1}^r [j_{2,k} d(k)] }}.
		\end{equation*}
		Since $d(1) = 1$ and $j_{1,k} = j_{2,k} \, \forall \, k = 2, \cdots , r$, we have $\alpha^{(l_1 - l_2) a i} \beta^{bj (j_{1,1} - j_{2,1})} = 1$ for all $b \equiv ax (\text{mod }t)$. So we have a representation, 
		\begin{equation*}
		\begin{split}
		f_1 \gamma_1 = l \gamma_0 + j_1 \gamma_1 \text{ where } &l \in \ZZ, \, 0 \leq |j_1| < \overline{m_1}\\
		&\alpha^{lai} \beta^{bj j_1} = 1 \, \forall \, b \equiv ax (\text{mod }t).
		\end{split}
		\end{equation*}
		$(f_1 - j_1) \gamma_1 = l \gamma_0 \Longrightarrow \overline{m_1} \mid (f_1 - j_1)$. Let $f_1 - j_1 = c \overline{m_1}$ where $c \in \ZZ$. Let $\overline{m_1} \gamma_1 = s \gamma_0$ where $s \in \ZZ_{> 0}$. Thus $f_1 \gamma_1 = cs \gamma_0 + j_1 \gamma_1 \Longrightarrow l \gamma_0 = cs \gamma_0 $. Thus $l = cs$. Since $\overline{m_1} \gamma_1 = s \gamma_0$, we have $Q_2 = Q_1^{\overline{m_1}} - \lambda X^s$ where $\lambda \in K \backslash \{ 0 \}$. $(\alpha^{ai}, \beta^{bj}) \cdot Q_2 = \beta^{bj \overline{m_1}} Q_1^{\overline{m_1}} - \lambda \alpha^{ais} X^s$. Since $Q_2$ is an eigenfunction we have, $\beta^{bj \overline{m_1}} = \alpha^{ais} \, \forall \, b \equiv ax (\text{mod }t)$. Again, from the above expression of $f_1 \gamma_1$, we have
		\begin{align*}
		& \alpha^{lai} \beta^{bj (f_1 - c \overline{m_1})} = 1 \, \forall \, b \equiv ax (\text{mod }t) \\
		&\Longrightarrow \alpha^{cs ai} \beta^{bj f_1} = \beta^{bjc \overline{m_1}} \, \forall \, b \equiv ax (\text{mod }t) \text{ as } l = cs\\
		&\Longrightarrow \beta^{bj f_1} = 1 \, \forall \, b \equiv ax (\text{mod }t) \Longrightarrow Nt \mid f_1 \Longrightarrow Nt \mid e.
		\end{align*}
		Thus we have obtained, $MNt \mid e \text{ as } MNt = N$.
		\par Now we consider the final case, $\gamma_0 \notin \Gamma_{\overline{ \nu}}$ and $\gamma_1 \notin \Gamma_{\overline{ \nu}}$. Let $f_0$ denote the order of $[\gamma_0]$ and $f_1$ denote the order of $[\gamma_1]$ in $\frac{\Gamma_\nu}{\Gamma_{\overline{ \nu}}}$. With the same arguments as before, we obtain $Mt \mid f_0 \text{ and } Nt \mid f_1$. Thus we have $Mt \mid e$ and $Nt \mid e$. Now $(Mt, Nt) = t$. So the lowest common multiple of $Mt$ and $Nt$ is $\frac{Mt Nt}{t} = MNt$. Thus, $MNt \mid e$.	
		\par Now, $K(X,Y)$ is a Galois extension of $Q(A_{i,j,t,x})$ with Galois group $H_{i,j,t,x}$ (Proposition $1.1.1$, [\ref{Benson}]). Thus $[K(X,Y) : Q(A_{i,j,t,x})] = |H_{i,j,t,x}| = MNt$ from Proposition \ref{order of group}. Let $\nu = \nu_1, \nu_2 , \cdots , \nu_r$ be all the distinct extensions of $\overline{ \nu} \text{ to } K(X,Y)$. Then (\S $12$, Theorem $24$, Corollary, [\ref{ZS}]), 
		\begin{equation*}
		efr = [K(X,Y) : Q(A_{i,j,t,x})] = MNt.
		\end{equation*}
		Since $MNt \mid e$, we have $e = MNt, \, r = 1$. So $\nu$ is the unique extension of $\overline{ \nu}$ to $K(X,Y)$. Thus $\overline{ \nu}$ does not split in $R_{\mathfrak{m}}$.	
		
	\end{proof}

\end{proof}


\begin{thebibliography}{1000000000}
	\bibitem{Ab1} S. Abhyankar, On the valuations centered in a local domain, Amer. J. Math. 78 (1956), 321 - 348. \label{Ab1}
	\bibitem{BEN9} D.J. Benson, Polynomial Invariants of Finite Groups, Cambridge University Press, 1993. \label{Benson}
	\bibitem{CUT_RAMVALN10} S.D. Cutkosky, Ramification of valuations and local rings in positive characteristic, Communications in Algebra 44 (2016), 2828-2866. \label{Cut_Ramification of Valuations}
	\bibitem{C2} S.D. Cutkosky, Finite generation of extensions of associated graded rings along a valuation, to appear in the Journal of the London Math. Soc. \label{C2}
	\bibitem{C1} S.D. Cutkosky, The role of defect and splitting in finite generation of extensions of associated graded rings along a valuation, Algebra and Number Theory 11 92017), 1461 - 1488. \label{C1}
	\bibitem{CV} S.D. Cutkosky and Pham An Vinh, Valuation semigroups of two dimensional  local rings, Proceedings of the London Mathematical Society 108 (2014), 350 - 384.\label{CV1}
	\bibitem{K} O. Kashcheyeva, Constructing examples of semigroups of valuations, J. Pure Appl. Algebra 200 (2016), 3826 - 3860.\label{K}
	\bibitem{Ku1} F.-V. Kuhlmann,   Valuation theoretic and model theoretic aspects of local uniformization, in Resolution of Singularities -
	A Research Textbook in Tribute to Oscar Zariski, H. Hauser, J. Lipman, F. Oort, A. Quiros (es.), Progress in Math. 181, Birkh\"auser (2000), 4559 - 4600. \label{Ku1}
	\bibitem{SL2} S. Lang, Algebra, revised third ed., Addison-Wesley Publishing Company Advanced Book Program, Reading, MA, 2002. \label{Lang}
	\bibitem{Mo} M. Moghaddam, A construction for a class of valuations of the field $K(X_1,\ldots, X_d,Y)$ with large value group, Journal of Algebra, 319, 7 (2008), 2803-2829. \label{Mo}
	\bibitem{NS} J. Novacoski and M. Spivakovsky, Key polynomials and pseudo-convergent sequences, J. Algebra 495 (2018), 199 - 219. \label{NS}
	\bibitem{JPS5} Jean-Pierre Serre, A Course In Arithmetic, Graduate Texts In Mathematics, 7 , New York - Heidelberg - Berlin, Springer-Verlag, 1973. \label{Serre}
	\bibitem{S} M. Spivakovsky, Valuations in function fields of surfaces, Amer. J. Math. 112 (1990), 107 - 156. \label{S}
	\bibitem{T1} B. Teissier,   Valuations, deformations and toric geometry, Valuation theory and its applications II,
	F.V. Kuhlmann, S. Kuhlmann and M. Marshall, editors, Fields
	Institute Communications 33 (2003), Amer. Math. Soc., Providence, RI, 361 \label{T1}
	-- 459.\bibitem{T2} B. Teissier, Overweight deformations of affine toric varieties and local uniformization, in Valuation theory in interaction, Proceedings of the second international conference on valuation theory, Segovia-El Escorial, 2011. Edited by A. Campillo, F-V- Kehlmann and B. Teissier. European Math. Soc. Publishing House, Congress Reports Series, Sept. 2014, 474 - 565.\label{T2}
	\bibitem{ZS7} O. Zariski and P. Samuel, Commutative Algebra, Volume II, Van Nostrand, 1960. \label{ZS}
\end{thebibliography}
\end{document}